\documentclass[a4paper,12pt,final]{amsart}
\usepackage[utf8]{inputenc}
\usepackage[T1]{fontenc}
\usepackage[UKenglish]{babel}
\usepackage[a4paper,margin=30mm]{geometry}
\usepackage{verbatim}

\allowdisplaybreaks
\usepackage{times}
\usepackage{dsfont,mathrsfs}
\usepackage{amsmath}
\usepackage{amsthm}
\usepackage{amssymb}
\usepackage{amsfonts}
\usepackage{latexsym}
\usepackage{booktabs}

\usepackage[colorlinks=true,
        linkcolor=black,
        citecolor=black,
        filecolor=blue,
       pagecolor=black,
        urlcolor=black,
        bookmarks=true,
        bookmarksopen=true,
        bookmarksopenlevel=3,
        plainpages=false,
        pdfpagelabels=true,
       pagebackref]{hyperref}

\usepackage{xcolor}

\newtheorem{theorem}{Theorem}[section]
\newtheorem{lemma}[theorem]{Lemma}
\newtheorem{proposition}[theorem]{Proposition}
\newtheorem{corollary}[theorem]{Corollary}

\theoremstyle{definition}

\newtheorem{example}[theorem]{Example}
\newtheorem{remark}[theorem]{Remark}
\numberwithin{equation}{section}
\renewcommand{\labelenumi}{\roman{enumi})}
\renewcommand\theenumi\labelenumi
\renewcommand{\leq}{\leqslant}
\renewcommand{\le}{\leqslant}
\renewcommand{\geq}{\geqslant}
\renewcommand{\ge}{\geqslant}

\newcommand{\real}{\mathds{R}}         
\newcommand{\nat}{\mathds{N}}

\newcommand{\Ee}{\mathds{E}}            
\newcommand{\Pp}{\mathds{P}}

\newcommand{\Bcal}{\mathcal{B}}
\newcommand{\Ccal}{\mathcal{C}}
\newcommand{\Mcal}{\mathcal{M}}

\newcommand{\dif}{\mathrm{d}}
\newcommand{\eup}{\mathrm{e}}
\newcommand{\I}{\mathds{1}}

\newcommand{\normal}{\color{black}}
\usepackage{marginnote}
\usepackage[notref,notcite]{showkeys}

\begin{document}
\title[Singular integrals of subordinators and SPDEs]{Singular integrals of subordinators with applications to structural properties of SPDEs}

\author[C.-S.~Deng]{Chang-Song Deng}
\address[C.-S.~Deng]{School of Mathematics and Statistics, Wuhan University, Wuhan 430072, China}
\email{dengcs@whu.edu.cn}

\author[R.L.~Schilling]{Ren\'{e} L.\ Schilling}
\address[R.L.\ Schilling]{%
TU Dresden, Fakult\"{a}t Mathematik, Institut f\"{u}r Ma\-the\-ma\-ti\-sche Sto\-cha\-stik, 01062 Dresden, Germany}
\email{rene.schilling@tu-dresden.de}

\author[L.~Xu]{Lihu Xu}
\address[L.~Xu]{Department of Mathematics, Faculty of Science and Technology, University of Macau, Macau S.A.R., China }
\email{lihuxu@um.edu.mo}

\keywords{subordinator; zero-one law; moment estimates; singular integral; stochastic convolution; invariant measures; accessibility; Galerkin approximation.}
\subjclass[2010]{60H15; 60G51; 60G52.}

\begin{abstract}
    We study stochastic integrals driven by a general subordinator and establish a zero-one law for the finiteness of the resulting integral as well as moment estimates. As an application, we use these results to obtain structural properties of SPDEs driven by multiplicative pure jump noise, which include (1) a maximal inequality for a multiplicative stochastic convolution $Z_t$, (2) a small ball probability of $Z_t$, (3) the existence of invariant measures and accessibility to zero of SPDEs, and (4) a Galerkin approximation of solutions to SPDEs.
\end{abstract}

\maketitle

\tableofcontents

\section{Introduction}\label{sub}

A subordinator $(S_t)_{t\geq 0}$ is an increasing L\'evy process on $[0,\infty)$ starting at $S_0=0$. As usual, we use a c\`adl\`ag (finite left limits, right-continuous) modification of $S_t$. The law of a subordinator is determined by the Laplace transform of the random variables $S_t$. Because of the independent and stationary increments property of a subordinator, its Laplace transform is of the form
\begin{gather*}
    \Ee\left[\eup^{-rS_t}\right] = \eup^{-t\phi(r)},\quad r>0,\,t\geq 0,
\end{gather*}
where the exponent $\phi:(0,\infty)\to (0,\infty)$ is a Bernstein function with $\phi(0+)=0$, i.e.\ a $C^\infty$-function such that $\phi\geq 0$ and with alternating derivatives $(-1)^{n+1} \phi^{(n)}\geq 0$, $n\in\nat$. Every such $\phi$ has a unique L\'{e}vy--Khintchine representation
\begin{gather*}
    \phi(r)
    = br+\int_{(0,\infty)}\left(1-\eup^{-rs}\right) \nu(\dif s)
\end{gather*}
with a drift parameter $b\geq 0$ and a L\'{e}vy measure $\nu$, i.e.\ a Radon measure on $(0,\infty)$ satisfying $\int_{(0,\infty)}(1\wedge s)\,\nu(\dif s)<\infty$. We use~\cite{SSV12} as a standard reference for Bernstein functions.

The parameters $b$ and $\nu$ also determine the structure of $S_t$ via the L\'evy--It\^o representation
\begin{gather*}
    S_t = bt + \sum_{0<r\leq t} \Delta S_r,
\end{gather*}
where $\Delta S_r = S_r - S_{r-}$ is the jump of $(S_t)_{t\geq 0}$ at time $t=r$. The jumps form a Poisson point process with intensity measure $\dif t\times \nu(\dif s)$; note that $S_0 = S_{0+}$, i.e.\ there is a.s.\ no instantaneous jump at time $t=0$. It is well known that $t\mapsto S_t$ is a.s.\ strictly increasing if $b>0$ or $\nu(0,\infty)=\infty$; this is also equivalent to saying that $\phi$ is an unbounded function.

Among the most important subordinators are the $\alpha$-stable subordinators $(0<\alpha<1)$ whose Bernstein functions are of the form $\phi(r) = r^\alpha$, i.e.\ $b=0$ and $\nu(\dif s) = \alpha \Gamma(1-\alpha)^{-1} s^{-\alpha-1}\,\dif s$.
We refer the reader to \cite{BoGrRy10, CHXZ17,ChKu03, De14,KiSoVo12,PaXiYa19,WaWa14} for results on $\alpha$-stable subordinate Brownian motion.

Since subordinators are a.s.\ increasing, we may use them as random time-changes of other stochastic processes. This procedure is called subordination (in the sense of Bochner) and it allows us to represent many L\'evy processes as time-changed (``subordinated'') Brownian motions; in this way we can get, for example, all symmetric stable L\'evy processes. Our standard reference for L\'evy processes and subordinators is~\cite{Ber96}.

We are interested in (stochastic) integrals of the following form
\begin{gather}\label{sub-e10}
    \int_0^\infty f(t)\,\dif S_t,
\end{gather}
where $f$ is a non-random integrand which may be a singular function, and we are going to establish a zero-one law for the finiteness and various moment formulas. A related zero-one
law for integral functionals of spectrally positive
L\'{e}vy processes can be found in \cite{LiZh18}.

As an application of these results we shall use them to study structural properties of SPDEs driven by multiplicative pure jump noise as follows. Let $(H, |\cdot|)$  be a separable Hilbert space, and $\mathds{W}=(W_t)_{t \ge 0}$ a cylindrical Wiener process on $H$ with filtration $(\mathscr{G}_t)_{t\geq 0}$, see e.g.\ \cite{DaZa92}. We consider the following SPDE:
\begin{equation} \label{int-e02}
\dif X_t=[-AX_t+F(X_t)]\,\dif t+Q(X_{t-})\,\dif W_{S_t}, \quad X_{0}=x \in H,
\end{equation}
where $\mathds{S}=(S_t)_{t \ge 0}$ is a subordinator with Bernstein function $\phi$. We assume that $\mathds S$ is independent of $\mathds{W}$; moreover we need:
\begin{align}
\label{A1}\tag{\bfseries A1}
    &\parbox[t]{.9\linewidth}{$Q:H \to  \mathcal L_{\mathrm{HS}}(H)$ is a bounded, Lipschitz-continuous function, taking values in the set $\mathcal L_{\mathrm{HS}}(H)$ of Hilbert--Schmidt operators on $H$, such that
    \begin{center}
            $\|Q\|_{\mathrm{HS},\infty}:=\sup_{x \in H} \|Q(x)\|_{\mathrm{HS}}<\infty$,\\[5pt]
            $\|Q(x)-Q(y)\|_{\mathrm{HS}} \leq C |x-y| \quad \forall x, y \in H.$
    \end{center}}
\\
\label{A2}\tag{\bfseries A2}
    &\parbox[t]{.9\linewidth}{$A$ is a self-adjoint operator such that there exists an orthonormal basis $(e_n)_{n\in\nat}$ of eigenvectors $A e_k=\gamma_k e_k$, $k \in \nat$, and the eigenvalues satisfy
    \begin{center}
    $
        0<\gamma_1 \leq \gamma_2 \leq \dots \leq \gamma_n \leq \dots, \qquad \lim_{n \to  \infty} \gamma_n=\infty.
    $
    \end{center}}
\\
\label{A3}\tag{\bfseries A3}
    &F:H \to  H \text{\ \ is a bounded, Lipschitz-continuous function.}
\end{align}
If $(\mathscr{F_t})_{t\geq 0}$ is the filtration generated by $(S_t)_{t\geq 0}$ and $(\mathscr{G}_{t})_{t\geq 0}$, then $(W_{S_t})_{t\geq 0}$ is adapted to the filtration $\mathscr{F}_{S_t} = \left\{F\in\mathscr{F}_\infty : F\cap\{S_t\leq r\}\in\mathscr{F}_r \text{\ \ for all $r\geq 0$}\right\}$.
By a standard Picard iteration argument, see e.g.\ \cite[Theorems 9.29, Theorem 9.34]{PeZa07}, the conditions \eqref{A1}--\eqref{A3} ensure that there is a unique $H$-valued c\`adl\`ag process $(X_t)_{t\geq 0}$ which is adapted to the filtration $\mathscr{F}_{S_t}$ of the driving noise $(W_{S_t})_{t\geq 0}$, and satisfies the SPDE
\begin{gather}\label{int-e10}
    X_{t}=\eup^{-tA}x+\int_0^t \eup^{-(t-r)A} F(X_r)\,\dif r+ \int_0^t \eup^{-(t-r)A} Q(X_{r-} ) \,\dif W_{S_r}.
\end{gather}

In order to study the SPDE \eqref{int-e02} we need to understand the following stochastic convolution
\begin{equation}\label{con-e02}
    Z_t
    = \int_0^t \eup^{-(t-r)A} Q(X_{r-})\,\dif W_{S_r}.
\end{equation}
Using our results on \eqref{sub-e10}, we will prove a maximal inequality and a small ball probability estimate of $Z_{t}$ in Section \ref{con}, which is crucial for the proof of structural properties of $X_{t}$ and approximation results. The structural properties include (1) a maximal inequality of multiplicative stochastic convolution $Z_t$ (Theorem~\ref{con-21}), (2) a small ball probability for $Z_t$ (Theorem~\ref{con-25}), (3) the existence of invariant measures and accessibility to zero of the SPDE (Theorem~\ref{inv-11}, Theorem~\ref{acc-15}), and (4) a Galerkin approximation for the solution of the SPDE (Theorem~\ref{gal-13}). For the study of structural properties of SPDEs driven by a pure jump noise, we refer the reader to \cite{PrShXuZa12,KhSc17,PeZa07, ZhZh15} and the references therein.

For the readers' convenience, let us briefly recall the following standard estimates which will be frequently used in the sequel. Denote by $\|A\| = \sup_{|x|\leq 1} |Ax|$ the operator norm induced by the Hilbert norm $|\cdot|$.
\begin{align}
\label{int-e20}
    \|A^\theta \eup^{-tA} \| &\leq C_\theta t^{-\theta}\quad\text{for all\ \ } \theta>0,\\
\label{int-e22}
    |A^\theta x| &\geq \gamma^\theta_1 |x| \quad\text{for all\ \ } x \in H,\\
\label{int-e24}
    \|\eup^{-tA}\| &\leq \eup^{-\gamma_1 t}.
\end{align}
The first inequality is from \cite[(3.2)]{WaXu18} or \cite[Lemma 2.3]{PrXuZa11}, the second and third inequalities are both due to the spectral gap of $A$. In fact, if $(e_n)_{n \in \mathds N}$ is an orthonormal basis of $H$, we have $x=\sum_{k\in \mathds N} a_k e_k$ with $a_k \in \mathds R$ for each $k$ and
\begin{gather*}
    |A^\theta x|^2
    =\left|\sum_{k \in \mathds N} a_k A^\theta e_k\right|^2
    =\left|\sum_{k \in \mathds N} \gamma_k^{\theta} a_k e_k\right|^2
    =\sum_{k \in \mathds N} |\gamma^\theta_k a_k|^2
    \geq \gamma^{2 \theta}_1 \sum_{k \in \mathds N} |a_k|^2
    =\gamma^{2 \theta}_1 |x|^2.
\end{gather*}
The other relations can be obtained by very similar arguments. As usual, we write $X\stackrel{d}{=} Y$ if the random variables $X$ and $Y$ have the same distribution.

\section{A zero-one law for integrals driven by subordinators}

Since the integrator $S_t$ is a.s.\ increasing, the integral \eqref{sub-e10} has a classical pathwise meaning as Lebesgue--Stieltjes integral and we can consider any measurable positive $f:(0,\infty)\to [0,\infty]$. The following simple but useful lemma on the characteristic functional of a subordinator will be crucial for our study.

\begin{lemma}[characteristic functional]\label{sub-21}
Let $(S_t)_{t\geq 0}$ be a subordinator with Bernstein function $\phi$. For any measurable and positive function  $f:(0,\infty)\to [0,\infty)$ the following equality holds:
\begin{gather*}
    \Ee\left(\exp\left[-\int_0^\infty f(t)\,\dif S_t\right]\right)
    =
    \exp\left[-\int_0^\infty\phi\left(f(t)\right) \dif t\right].
\end{gather*}
\end{lemma}

\begin{proof}
Assume first that $f(t)$ is a step function of the form $\sum_{i=1}^n f_{i-1} \I_{(t_{i-1},t_{i}]}(t)$, $f_i\geq 0$, $0\leq t_0 < t_1 < \dots < t_n < \infty$. Using the fact that a subordinator has stationary and independent increments gives
\begin{align*}
    &\Ee\left(\exp\left[-\int_0^\infty f(t)\,\dif S_t\right]\right)
    = \Ee\left(\exp\left[-\sum_{i=1}^n f_{i-1}\left(S_{t_i}-S_{t_{i-1}}\right)\right]\right)\\
    &\quad=\prod_{i=1}^n \Ee\left[\eup^{-f_{i-1}S_{t_i-t_{i-1}}}\right]
    =\prod_{i=1}^n \eup^{-\left(t_i-t_{i-1}\right)\phi\left(f_{i-1}\right)}\\
    &\quad=\exp\left[-\sum_{i=1}^n\phi(f_{i-1})\left(t_i-t_{i-1}\right)\right]
    =\exp\left[-\int_0^\infty\phi\big(f(t)\big)\,\dif t\right].
\end{align*}

If $f$ is a general positive measurable function such that $\int_0^\infty \phi(f(t))\,\dif t < \infty$, we can approximate $f$, hence $\phi\circ f$, in $L^1((0,\infty);\dif t)$-sense by step functions as above, and the claim follows by a standard density argument. If $\int_0^\infty \phi(f(t))\,\dif t = \infty$, we approximate $\phi$ by the increasing sequence $\phi_n(t) := \min\{\phi(t),n\} \I_{[0,n]}(t)$. Since $\int_0^\infty \phi_n(f(t))\,\dif t \leq n^2$, we approximate this, as before, by step functions and use a monotone convergence theorem.
\end{proof}

The first application of the characteristic functional is the following result on time reversals. Throughout the paper we will need the following elementary identities
\begin{gather}\label{sub-e14}
    s^p
    =
    \frac{p}{\Gamma(1-p)}\int_0^\infty \left(1-\eup^{-sr}\right)\frac{\dif r}{r^{p+1}},
    \quad s\geq 0,\;p\in(0,1),
    \\
    \label{sub-e16}
    s^p
    = \frac{1}{\Gamma(-p)}\int_0^\infty\eup^{-sr}\,\frac{\dif r}{r^{p+1}},
    \quad s>0,\; p<0.
\end{gather}

\begin{corollary}[time reversal]\label{sub-23}
    Let $(S_t)_{t\geq 0}$ be a subordinator with Bernstein function $\phi$, $T>0$, and $f:(0,\infty)\to [0,\infty)$ a measurable function. For any $-\infty < p<1$ it holds that
    \begin{gather*}
        \Ee\left[\left(\int_0^Tf(T-t)\,\dif S_t\right)^p\right]
        = \Ee\left[\left(\int_0^Tf(t)\,\dif S_t\right)^p\right] \in [0,+\infty].
    \end{gather*}
\end{corollary}
\begin{proof}
    The case $p=0$ is trivial. Since
    \begin{gather*}
        \int_0^T\phi\big(f(T-t)\big)\,\dif t
        = \int_0^T\phi\big(f(t)\big)\,\dif t,
    \end{gather*}
    the assertion follows immediately from Lemma~\ref{sub-21}, Tonelli's theorem, and the identities \eqref{sub-e14}, \eqref{sub-e16}.
\end{proof}
In the following two sections we will obtain conditions ensuring the finiteness of the moments appearing in Corollary~\ref{sub-23}.

If $f:(0,\infty)\to [0,\infty)$ is a bounded measurable function, the integral $\int_0^\infty f(t)\,\dif S_t$ is finite if, and only if, the tail integrals $\int_n^\infty f(t)\,\dif S_t$, $n\in\nat$, are finite. This means that the set $\left\{\omega : \int_0^\infty f(t)\,\dif S_t(\omega) < \infty\right\}$ is a terminal event for the natural filtration of $(S_t)_{t\geq 0}$, hence it has probability either $0$ or $1$ by Kolmogorov's zero-one law. The following result contains both a generalization (to all positive $f$) and a criterion to decide whether the probability is $1$.

\begin{proposition}[zero-one law]\label{sub-25}
    Let $(S_t)_{t\geq 0}$ be a subordinator with Bernstein function $\phi$ and $f:(0,\infty)\to [0,\infty)$ a measurable function. The following assertions are equivalent:
\begin{enumerate}\setlength{\itemsep}{8pt}
    \item\label{sub-25-i}
    $\displaystyle\Pp\left(\int_0^\infty f(t)\,\dif S_t<\infty\right)>0$.

    \item\label{sub-25-ii}
    $\displaystyle\Pp\left(\int_0^\infty f(t)\,\dif S_t<\infty\right)=1$.

    \item\label{sub-25-iii}
    $\displaystyle\int_0^\infty\phi\big(f(t)\big)\,\dif t<\infty$.
\end{enumerate}
\end{proposition}

\begin{proof}
\ref{sub-25-iii}~$\Rightarrow$~\ref{sub-25-ii}: If we use Lemma~\ref{sub-21} with $f$ replaced by $\lambda f$ for some $\lambda>0$ and combine it with the monotone convergence theorem we get
\begin{align*}
    \Pp\left(\int_0^\infty f(t)\,\dif S_t<\infty\right)
    &= \lim_{\lambda\to 0}\Ee\left(\exp\left[-\lambda\int_0^\infty f(t)\,\dif S_t\right]\I_{\left\{\int_0^\infty f(t)\,\dif S_t<\infty\right\}}\right)\\
    &= \lim_{\lambda\to 0}\Ee\left(\exp\left[-\lambda\int_0^\infty f(t)\,\dif S_t\right]\right)\\
    &=\lim_{\lambda\to 0}\exp\left[-\int_0^\infty\phi\big(\lambda f(t)\big)\,\dif t\right] \:=\: 1.
\end{align*}

\noindent
The direction~\ref{sub-25-ii}~$\Rightarrow$~\ref{sub-25-i} is obvious, and~\ref{sub-25-i}~$\Rightarrow$~\ref{sub-25-iii} follows thus: Suppose that $\int_0^\infty\phi\big(f(t)\big)\,\dif t = \infty$. By Lemma~\ref{sub-21},
\begin{gather*}
    \Ee\left(\exp\left[-\int_0^\infty f(t)\,\dif S_t\right]\right)=0,
    \quad\text{hence}\quad
    \Pp\left(\int_0^\infty f(t)\,\dif S_t=\infty\right)=1,
\end{gather*}
which contradicts~\ref{sub-25-i}. This completes the proof.
\end{proof}

\section{Moment formulas for singular integrals driven by a stable subordinator}\label{subsec:stable}

Throughout this section $(S_t)_{t\geq 0}$ is an $\alpha$-stable subordinator; the corresponding Bernstein function is of the form $\phi(r)=r^\alpha$, $\alpha\in (0,1)$.
For the special case with $f(t)=t^{-\theta}, \theta \in (0,\infty)$ in Proposition \ref{sub-31} below, moment estimates have been established in \cite{Xu18}.
\begin{proposition}\label{sub-31}
    Let $S_t$ be an $\alpha$-stable subordinator, $0<\alpha<1$. If $f:(0,\infty)\to [0,\infty)$ is a measurable  function such that $\mathrm{Leb}\{f>0\}>0$, then
    \begin{gather*}
        \Ee\left[\left(\int_0^\infty f(t)\,\dif S_t\right)^p\right]
        =
        \begin{cases}
        \displaystyle\frac{\Gamma\left(1-\frac{p}{\alpha}\right)}{\Gamma(1-p)}\left(\int_0^\infty f(t)^\alpha\,\dif t\right)^{\frac{p}{\alpha}},
        &\text{if $-\infty < p<\alpha$},\\
        \displaystyle\infty,
        &\text{if $p\geq\alpha$}.
        \end{cases}
    \end{gather*}
\end{proposition}

\begin{proof}
Without loss of generality we may assume that $0<\int_0^\infty f(t)^\alpha\,\dif t<\infty$ and
$p\neq 0$. We distinguish between three cases.

\medskip\noindent
\emph{Case 1}: $0<p<1$. Combining the elementary
identity \eqref{sub-e14} with Tonelli's theorem and Lemma~\ref{sub-21}, yields
\begin{align*}
    \Ee\left[\left(\int_0^\infty f(t)\,\dif S_t\right)^p\right]
    &=\frac{p}{\Gamma(1-p)}\,\Ee\left[\int_0^\infty \left(1-\eup^{-r\int_0^\infty f(t)\,\dif S_t}\right) \frac{\dif r}{r^{p+1}}\right]\\
    &=\frac{p}{\Gamma(1-p)}\int_0^\infty \left(1-\eup^{-r^\alpha\int_0^\infty f(t)^\alpha\,\dif t}\right)
    \frac{\dif r}{r^{p+1}}.
\end{align*}
If we change variables according to $s=r^\alpha\int_0^\infty f(t)^\alpha\,\dif t$ and use~\eqref{sub-e14} once again, we obtain
 \begin{align*}
    \Ee\left[\left(\int_0^\infty f(t)\,\dif S_t\right)^p\right]
    &=\frac{p}{\alpha\Gamma(1-p)}\left(\int_0^\infty f(t)^\alpha\,\dif t\right)^{\frac{p}{\alpha}}\int_0^\infty\left(1-\eup^{-s}\right)
    s^{-\frac{p}{\alpha}-1}\,\dif s\\
    &=
    \begin{cases}
    \displaystyle\frac{\Gamma\left(1-\frac{p}{\alpha}\right)}{\Gamma(1-p)}\left(
    \int_0^\infty f(t)^\alpha\,\dif t\right)^{\frac{p}{\alpha}},
    &\text{if $p\in(0,\alpha)$},\\
    \displaystyle\infty,
    &\text{if $p\in[\alpha,1)$}.
    \end{cases}
\end{align*}

\medskip\noindent
\emph{Case 2}: $p\geq 1$. It follows from Jensen's inequality and the first case that
\begin{gather*}
    \Ee\left[\left(\int_0^\infty f(t)\,\dif S_t\right)^p\right]
    \geq
    \left(\Ee\left[\left(\int_0^\infty f(t)\,\dif S_t\right)^\alpha\right]\right)^{\frac{p}{\alpha}}
    =\infty.
\end{gather*}

\medskip\noindent
\emph{Case 3}: $p<0$. We use the identity
\eqref{sub-e16}, Tonelli's theorem
and Lemma~\ref{sub-21} to get
\begin{align*}
    \Ee\left[\left(\int_0^\infty f(t)\,\dif S_t\right)^p\right]
    &= \frac{1}{\Gamma(-p)}\,\Ee\left[\int_0^\infty \eup^{-r\int_0^\infty f(t)\,\dif S_t}\,\frac{\dif r}{r^{p+1}}\right]\\
    &= \frac{1}{\Gamma(-p)}\int_0^\infty\eup^{-r^\alpha\int_0^\infty f(t)^\alpha\,\dif t}\,\frac{\dif r}{r^{p+1}}\\
    &= \frac{1}{\alpha\Gamma(-p)}\left(\int_0^\infty f(t)^\alpha\,\dif t\right)^{\frac{p}{\alpha}}
    \int_0^\infty\eup^{-s}s^{-\frac{p}{\alpha}-1}\,\dif s\\
    &= \frac{\Gamma\left(-\frac{p}{\alpha}\right)}{\alpha\Gamma(-p)}\left(\int_0^\infty f(t)^\alpha\,\dif t\right)^{\frac{p}{\alpha}}\\
    &= \frac{\Gamma\left(1-\frac{p}{\alpha}\right)}{\Gamma(1-p)}\left(\int_0^\infty f(t)^\alpha\,\dif t\right)^{\frac{p}{\alpha}};
\end{align*}
in the last equality we use the functional equation $\Gamma(1+r)=r\Gamma(r)$ of the Gamma-function.
\end{proof}

\begin{corollary}\label{sub-33}
    Let $S_t$ be an $\alpha$-stable subordinator, $0<\alpha<1$, $p,\theta\in\real$ and $T>0$.
    \begin{enumerate}
    \item\label{sub-33-i}
    According to $\theta< \frac 1\alpha$ or $\theta\geq \frac 1\alpha$ one has with probability one
    \begin{gather*}
        \int_0^Tt^{-\theta}\,\dif S_t<\infty,
        \quad\text{resp.,}\quad
        =\infty,
    \end{gather*}
    and
    \begin{gather*}
        \Ee\left[\left(\int_0^Tt^{-\theta}\,\dif S_t\right)^p\right]
        =
        \begin{cases}
            \displaystyle\frac{\Gamma\left(1-\frac{p}{\alpha}\right)}{(1-\alpha\theta)^{\frac{p}{\alpha}}\Gamma(1-p)}\,
            T^{p(\frac 1\alpha-\theta)},
            &\text{if $\theta < \frac 1\alpha$ \& $p<\alpha$},\\
            0,
            &\text{if $\theta\geq \frac 1\alpha$ \& $p<0$},\\
            1,
            &\text{if $\theta\geq \frac 1\alpha$ \& $p=0$}\\
            \infty,
            &\text{if $\theta\geq \frac 1\alpha$ \& $p>0$}.
        \end{cases}
    \end{gather*}

    \item\label{sub-33-ii}
    According to $\theta > \frac 1\alpha$ or $\theta\leq \frac 1\alpha$ one has with probability one
    \begin{gather*}
        \int_T^\infty t^{-\theta}\,\dif S_t<\infty,
        \quad\text{resp.,}\quad
        =\infty,
    \end{gather*}
    and
    \begin{gather*}
        \Ee\left[\left(\int_T^\infty t^{-\theta}\,\dif S_t\right)^p\right]
        =
        \begin{cases}
        \displaystyle
        \frac{\Gamma\left(1-\frac{p}{\alpha}\right)}{(\alpha\theta-1)^{\frac{p}{\alpha}}\Gamma(1-p)}\,T^{p(\frac 1\alpha-\theta)},
        &\text{if $\theta> \frac 1\alpha$ \& $p<\alpha$},\\
        0,
        &\text{if $\theta\leq \frac 1\alpha$ \& $p<0$},\\
         1,
        &\text{ if $\theta\leq \frac 1\alpha$ \& $p=0$},\\
         \infty,
        &\text{ if $\theta\leq \frac 1\alpha$ \& $p>0$}.
        \end{cases}
    \end{gather*}

    \item\label{sub-33-iii} For all $\lambda>0$ one has
    \begin{gather*}
        \Ee\left[\left(\int_0^T\eup^{-\lambda t}\,\dif S_t\right)^p\right]
        =
        \begin{cases}
        \displaystyle
        \frac{\Gamma\left(1-\frac{p}{\alpha}\right)}{\Gamma(1-p)}
        \left(\frac{1-\eup^{-\alpha\lambda T}}{\alpha\lambda}\right)^{\frac{p}{\alpha}},
        &\text{if $p<\alpha$},\\
        \infty,
        &\text{if $p\geq\alpha$}.
        \end{cases}
    \end{gather*}
    \end{enumerate}
\end{corollary}

\begin{proof}
    The assertions~\ref{sub-33-i} and~\ref{sub-33-ii} follow from Lemma~\ref{sub-21} and Proposition~\ref{sub-31} with $f(t)=t^{-\theta}\I_{(0,T)}(t)$ and
    $f(t)=t^{-\theta}\I_{(T,\infty)}(t)$. In a similar way~\ref{sub-33-iii} can be obtained from Proposition~\ref{sub-31} if we use $f(t)=\eup^{-\lambda t}\I_{(0,T)}(t)$.
\end{proof}

\section{Moment estimates for singular integrals driven by a general subordinator}\label{subsec:general}

We will now consider a subordinator $(S_t)_{t\geq 0}$ with Bernstein function $\phi$. Other than in the stable case, we cannot hope for exact moment formulae. Therefore we aim for estimates of the following type:
\begin{align}
\label{sub-e40}
    \Ee\left[\left(\int_0^T t^{-\theta}\,\dif S_t\right)^p\right]
    &\leq CT^{-p\theta}\left[\phi^{-1}\left(\frac{1}{T}\right)\right]^{-p},
\\\label{sub-e42}
    \Ee \left[S_T^p\right]
    &\leq C \left[\phi^{-1}\left(\frac{1}{T}\right)\right]^{-p},
\\\label{sub-e44}
    \Ee\left[\left(\int_0^T \eup^{-\lambda t}\,\dif S_t\right)^p\right]
    &\leq C\left[\phi^{-1}\left(\frac{1}{T\wedge1}\right)\right]^{-p},
\end{align}
with constants $C$ depending on $p\in\real$, $\theta\geq 0$ and $\lambda>0$.

\begin{proposition}\label{sub-41}
Let $S_t$ be a subordinator with Bernstein function $\phi$.
\begin{enumerate}
\item\label{sub-41-i}
    The estimate~\eqref{sub-e40} holds for $p\leq 0$, $\theta\geq 0$ and all $T\in [1,\infty)$, if
    \begin{gather*}
        \liminf_{s\to \infty}\frac{\phi(s)}{\log s}>0
        \quad \text{and}\quad
        \liminf_{s\to 0}\frac{\phi(2s)}{\phi(s)}>1.
    \end{gather*}

\item\label{sub-41-ii}
    The estimate~\eqref{sub-e40} holds for $p\leq 0$, $\theta\geq 0$ and all $T\in (0,1]$, if
    \begin{gather*}
        \liminf_{s\to \infty}\frac{\phi(2s)}{\phi(s)}>1.
    \end{gather*}

\item\label{sub-41-iii}
    The estimate~\eqref{sub-e42} holds for all $T>0$  \textup{[}resp.\ $T\geq 1$\textup{]} if
    \begin{equation}\label{sub-e48}
        0 \leq p < \log_2\left(\inf_{s>0}\frac{\phi(2s)}{\phi(s)}\right)
        \qquad
        \left[\text{resp.}
        \quad 0 \leq p < \log_2\left(\liminf_{s\to 0}\frac{\phi(2s)}{\phi(s)}\right)
        \right].
    \end{equation}

\item\label{sub-41-iv}
    The estimate~\eqref{sub-e40} holds for all $T\geq 1$ if
    \begin{gather}\label{sub-e50}
        0 \leq p < \log_2\left(\liminf_{s\to 0}\frac{\phi(2s)}{\phi(s)}\right)
    \quad\text{and}\quad
        0 \leq \theta < \left[\log_2\left(\sup_{s>0}\frac{\phi(2s)}{\phi(s)}\right)\right]^{-1}.
    \end{gather}

\item\label{sub-41-v}
    The estimate~\eqref{sub-e40} holds for all $T\in(0,1]$ if
    \begin{equation}\label{sub-e52}
        0 \leq p < \log_2\left(\inf_{s>0}\frac{\phi(2s)}{\phi(s)}\right)
        \quad\text{and}\quad
        0 < \theta <\left[\log_2\left(\limsup_{s\to \infty}\frac{\phi(2s)}{\phi(s)}\right)\right]^{-1}.
    \end{equation}

\item\label{sub-41-vi}
    The estimate~\eqref{sub-e44} holds for all $T>0$, $\lambda>0$ and $p<0$, if
    \begin{gather*}
        \liminf_{s\to \infty}\frac{\phi(2s)}{\phi(s)}>1.
    \end{gather*}

\item\label{sub-41-vii}
    If $p>0$ and $\lambda>0$, then
    \begin{equation}\label{sub-e54}
        \int_0^1\frac{\phi(s)}{s^{p+1}}\,\dif s<\infty
        \iff
        \Ee\left[\left(\int_0^\infty \eup^{-\lambda t}\,\dif S_t\right)^p\right] < \infty.
    \end{equation}
\end{enumerate}
\end{proposition}

Before we are going to prove Proposition~\ref{sub-41} we will add a few remarks on the assumptions made in this proposition and give some examples.

\begin{remark}\label{sub-43}
\begin{enumerate}
\item\label{sub-43-a}
Corollary~\ref{sub-33} shows that all assertions of Proposition~\ref{sub-41} are sharp for $\alpha$-stable subordinators.

\item\label{sub-43-b}
Since Bernstein functions are subadditive, we have $\phi(2s)\leq 2\phi(s)$ for all $s>0$. This means that both~\eqref{sub-e48} and the first condition in~\eqref{sub-e50} imply $p\in[0,1)$.

\item\label{sub-43-c}
Since Bernstein functions are concave, we get
\begin{gather*}
    \phi'(s)\leq\frac{\phi(s)}{s},\quad s>0,
\end{gather*}
and, therefore,
\begin{gather*}
    \int_0^1\frac{\phi(s)}{s^{p+1}}\,\dif s
    \geq\int_0^1\frac{\phi'(s)}{s^{p}}\,\dif s
    \geq\phi'(1)\int_0^1\frac{\dif s}{s^p}.
\end{gather*}
This means that~\eqref{sub-e54} can only happen if $p<1$.

\item\label{sub-43-d}
The condition~\eqref{sub-e48}
implies that there is some $\tilde{p}>p$ such that for all $s>0$
\begin{gather*}
    \frac{\phi(2s)}{\phi(s)}
    >
    2^{\tilde{p}}
    \quad\text{and for all $k\in\nat\cup\{0\}$},\quad
    \phi\left(2^{-k}s\right)
    \leq
    2^{-k\tilde{p}}\phi(s).
\end{gather*}
A routine monotonicity argument shows that this implies
\begin{equation}\label{sub-e56}
    \phi\left(2^{-x}s\right)
    \leq
    2^{\tilde p} 2^{-x\tilde{p}}\phi(s)\quad\text{for all $x\geq 0$ and $s>0$}.
\end{equation}
Under the alternative condition, this estimate is still valid for \emph{small} values $0<s<s_0$.

\item\label{sub-43-e}
The second condition in~\eqref{sub-e50} implies that there is some $0<\tilde{\theta}<1/\theta$ such that
\begin{gather*}
    \phi\left(2^ks\right)\leq2^{\tilde{\theta}k}
    \phi(s)\quad\text{for all $k\in\nat$ and $s>0$}.
\end{gather*}
A routine monotonicity argument shows that this implies
\begin{equation}\label{sub-e58}
    \phi\left(2^xs\right)
    \leq 2^{1/\theta} 2^{\tilde{\theta}x}
    \phi(s)\quad\text{for all $x\geq 0$ and $s>0$}.
\end{equation}
If we assume, instead, the weaker second condition in~\eqref{sub-e52}, the estimate~\eqref{sub-e58} is still valid for \emph{large} values $s > s_0$.
\end{enumerate}
\end{remark}

\begin{example}\label{sub-45}
    From \cite[Proposition~7.16(ii)]{SSV12} and \cite[table entry~16.2.6]{SSV12} we know that the functions
    \begin{align*}
        \phi(s) &= s^\alpha \log^\beta (1+s),    && 0<\alpha<1,\; 0\leq \beta \leq 1-\alpha\\
        \psi(s) &= s^\alpha \log^{-\beta} (1+s),    &&  0 \leq \beta \leq \alpha < 1\\
        \omega(s) &= s(1+s)^{-\alpha}, && 0 < \alpha < 1
    \end{align*}
    are (complete) Bernstein functions. The results of Proposition~\ref{sub-41} are summarized for these functions in Table~\ref{exa-tab}.

\begin{table}[ht]\centering
    \renewcommand*{\arraystretch}{1.3}
    \caption{Overview of the results of Proposition~\ref{sub-41} for some concrete examples.}\label{exa-tab}
    \begin{tabular}{@{}p{.17\linewidth}p{.25\linewidth}p{.25\linewidth}p{.25\linewidth}@{}} \toprule
    Estimate                     & $s^\alpha\log^\beta(1+s)$    & $s^\alpha\log^{-\beta}(1+s)$ & $s(1+s)^{-\alpha}$\\ \midrule
    \eqref{sub-e40},~~$T>0$       & $p\leq 0$,~~$\theta\geq 0$   & $p\leq 0$,~~$\theta\geq 0$   & $p\leq 0$,~~$\theta\geq 0$\\ \midrule
    \eqref{sub-e40},~~$T\geq 1$   & $0\leq p<\alpha+\beta$       & $0\leq p<\alpha-\beta$       & $0\leq p<1$\\
                            & $0\leq\theta<(\alpha+\beta)^{-1}$ & $0\leq\theta<\alpha^{-1}$    & $0\leq\theta<1$\\ \midrule
    \eqref{sub-e40},~~$T\leq 1$   & $0\leq p<\alpha$             & $0\leq p<\alpha-\beta$       & $0\leq p<1-\alpha$\\
                                 & $0\leq\theta<\alpha^{-1}$       & $0\leq\theta<\alpha^{-1}$    & $0\leq\theta<(1-\alpha)^{-1}$\\ \midrule
    \eqref{sub-e42},~~$T>0$     & $0\leq p<\alpha$             & $0\leq p<\alpha-\beta$       & $0\leq p<1-\alpha$\\ \midrule
    \eqref{sub-e44},~~$T>0$     & $p\leq 0$,~~$\lambda>0$      & $p\leq 0$,~~$\lambda>0$      & $p\leq 0$,~~$\lambda>0$\\ \midrule
    \eqref{sub-e54} applies      & $0\leq p<\alpha+\beta$       & $0\leq p<\alpha-\beta$       & $0\leq p<1$\\
                                 & $\lambda>0$                  & $\lambda>0$                  & $\lambda>0$\\ \bottomrule
    \end{tabular}
\end{table}
\end{example}

\begin{proof}[Proof of Proposition~\ref{sub-41}]
\ref{sub-41-i} \&~\ref{sub-41-ii}: Since $p<0$, the monotonicity of the integral gives
\begin{gather*}
    \Ee\left[\left(\int_0^T t^{-\theta}\,\dif S_t\right)^p\right]
    \leq \Ee\left[\left(\int_0^T T^{-\theta}\,\dif S_t\right)^p\right]
    = T^{-\theta p} \Ee\left[S_T^p\right],
\end{gather*}
so~\ref{sub-41-i} and~\ref{sub-41-ii} follow from the moment estimates in~\cite[Theorem 2.1(ii)]{DSS17}.

\medskip\ref{sub-41-iii}
By~\eqref{sub-e14}, Tonelli's theorem, Lemma~\ref{sub-21}, and the inequality $1-\eup^{-r}\leq1\wedge r$, $r\geq 0$,
we get for any $p\in(0,1)$ and $T>0$
\begin{equation}\label{sub-e60}
\begin{aligned}
    \Gamma(1-p)\Ee\left[S_T^p\right]
    &= p\Ee \left[\int_0^\infty \left(1-\eup^{-rS_T}\right) \frac{\dif r}{r^{p+1}}\right]\\
    &= p\int_0^\infty\left(1-\eup^{-T\phi(r)}\right) \frac{\dif r}{r^{p+1}}\\
    &\leq pT\int_0^{\phi^{-1}\left(\frac{1}{T}\right)}\phi(r)\,\frac{\dif r}{r^{p+1}}
    + p\int_{\phi^{-1}\left(\frac{1}{T}\right)}^\infty \frac{\dif r}{r^{p+1}}\\
    &= pT\int_0^{\phi^{-1}\left(\frac{1}{T}\right)} \phi(r)\frac{\dif r}{r^{p+1}}
    + \left[\phi^{-1}\left(\frac{1}{T}\right)\right]^{-p}.
\end{aligned}
\end{equation}
Since $\phi(0+)=0$, we get using integration
by parts,
\begin{equation}\label{sub-e62}
\begin{aligned}
    p\int_0^{\phi^{-1}\left(\frac{1}{T}\right)}\phi(r)\,\frac{\dif r}{r^{p+1}}
    &=\int_0^{\phi^{-1}\left(\frac{1}{T}\right)} s^{-p}\,\dif\phi(s)-\frac{1}{T}\left[\phi^{-1}\left(\frac{1}{T}\right)\right]^{-p}.
\end{aligned}
\end{equation}
Note that
\begin{equation}\label{sub-e64}
\begin{aligned}
    \int_0^{\phi^{-1}\left(\frac{1}{T}\right)}s^{-p}\,\dif\phi(s)
    &= \sum_{k=0}^\infty \int_{2^{-(k+1)}\phi^{-1}\left(\frac{1}{T}\right)}^{2^{-k}\phi^{-1} \left(\frac{1}{T}\right)}s^{-p}\,\dif\phi(s)\\
    &\leq \sum_{k=0}^\infty\left(\frac{1}{2^{k+1}}\phi^{-1}\left(\frac{1}{T}\right)\right)^{-p}
    \phi\left(\frac{1}{2^k}\phi^{-1}\left(\frac{1}{T}\right)\right)\\
    &= 2^p\left[\phi^{-1}\left(\frac{1}{T}\right)\right]^{-p}
    \sum_{k=0}^\infty2^{pk}\phi\left(\frac{1}{2^k}\phi^{-1}\left(\frac{1}{T}\right)\right).
\end{aligned}
\end{equation}
Combining~\eqref{sub-e60},~\eqref{sub-e62} and~\eqref{sub-e64}, we get for any $p\in(0,1)$ and $T>0$,
\begin{equation}\label{sub-e66}
    \Gamma(1-p)\Ee\left[S_T^p\right]
    \leq
    2^pT \left[\phi^{-1}\left(\frac{1}{T}\right)\right]^{-p}\sum_{k=0}^\infty2^{pk} \phi\left(\frac{1}{2^k}\phi^{-1}\left(\frac{1}{T}\right)\right).
\end{equation}
Since we assume~\eqref{sub-e48}, we may use the estimate from Remark~\ref{sub-43}.\ref{sub-43-d} in~\eqref{sub-e66}, and this gives for all $T>0$  [resp.\ $T\geq 1$]
\begin{equation}\label{sub-e68}
    \Gamma(1-p)\Ee\left[S_T^p\right]
    \leq
    \left(2^p\sum_{k=0}^\infty2^{-(\tilde{p}-p)k}\right)\left[\phi^{-1}\left(\frac{1}{T}\right)\right]^{-p}.
\end{equation}

\medskip\noindent\ref{sub-41-iv}
If $\theta=0$, we are in the situation of part~\ref{sub-41-iii} with $T\geq 1$.

Assume that $0<\theta < 1/p$. As in~\eqref{sub-e60}, we have for any $p\in(0,1)$ and $T>0$
\begin{align*}
    &\Gamma(1-p) \Ee\left[\left(\int_0^T t^{-\theta}\,\dif S_t\right)^p\right]\\
    &\quad= p\int_0^\infty\left(1-\exp\left[-\int_0^T\phi(rt^{-\theta})\,\dif t\right]\right)\frac{\dif r}{r^{p+1}}\\
    &\quad\leq p\int_0^{T^\theta\phi^{-1}\left(\frac{1}{T}\right)}\left(\int_0^T\phi(rt^{-\theta})\,\dif t\right)\frac{\dif r}{r^{p+1}}
        + p\int_{T^\theta\phi^{-1}\left(\frac{1}{T}\right)}^\infty\frac{\dif r}{r^{p+1}}\\
    &\quad= \frac{p}{\theta}\int_0^{T^\theta\phi^{-1}\left(\frac{1}{T}\right)}\left(\int_{rT^{-\theta}}^\infty\phi(s)\, \frac{\dif s}{s^{\frac1\theta+1}}\right)\frac{\dif r}{r^{-\frac1\theta+p+1}}
        + T^{-p\theta}\left[\phi^{-1}\left(\frac{1}{T}\right)\right]^{-p}\\
    &\quad= \frac{p}{\theta}\int_0^{\phi^{-1}\left(\frac{1}{T}\right)}\left(\int_0^{T^\theta s}\frac{\dif r}{r^{-\frac1\theta+p+1}}\right)\phi(s)
    \,\frac{\dif s}{s^{\frac1\theta+1}}\\
    &\qquad\mbox{}+\frac{p}{\theta}\int_{\phi^{-1}\left(\frac{1}{T}\right)}^\infty\left(\int_0^{T^\theta\phi^{-1}\left(\frac{1}{T}\right)}
    \frac{\dif r}{r^{-\frac1\theta+p+1}}\right)\phi(s)\,\frac{\dif s}{s^{\frac1\theta+1}}
    + T^{-p\theta}\left[\phi^{-1}\left(\frac{1}{T}\right)\right]^{-p}\\
    &\quad= \frac{T^{1-p\theta}}{1-p\theta}\,p\int_0^{\phi^{-1}\left(\frac{1}{T}\right)}\phi(s)\,\frac{\dif s}{s^{p+1}}\\
    &\qquad\mbox{}+\frac{T^{1-p\theta}}{1-p\theta}\,p\left[\phi^{-1}\left(\frac{1}{T}\right)\right]^{\frac1\theta-p}
    \int_{\phi^{-1}\left(\frac{1}{T}\right)}^\infty\phi(s)\,\frac{\dif s}{s^{\frac1\theta+1}}
    + T^{-p\theta}\left[\phi^{-1}\left(\frac{1}{T}\right)\right]^{-p}.
\end{align*}
In order to estimate the middle term in the above expression, we use integration by parts and get
\begin{align*}
    \int_{\phi^{-1}\left(\frac{1}{T}\right)}^\infty\phi(s)\,\frac{\dif s}{s^{\frac1\theta+1}}
    &= \theta\int_{\phi^{-1}\left(\frac{1}{T}\right)}^\infty r^{-\frac1\theta}\,\dif\phi(r)
    + \frac{\theta}{T}\left[\phi^{-1}\left(\frac{1}{T}\right)\right]^{-\frac1\theta}\\
    &= \theta\sum_{k=0}^\infty\int_{2^k\phi^{-1}\left(\frac{1}{T}\right)}^{2^{k+1}\phi^{-1}\left(\frac{1}{T}\right)} r^{-\frac1\theta}\,\dif\phi(r)
    +\frac{\theta}{T}\left[\phi^{-1}\left(\frac{1}{T}\right)\right]^{-\frac1\theta}\\
    &\leq \theta\left[\phi^{-1}\left(\frac{1}{T}\right)\right]^{-\frac1\theta}
    \sum_{k=0}^\infty2^{-\frac{k}{\theta}} \phi\left(2^{k+1}\phi^{-1}\left(\frac{1}{T}\right)\right)
    + \frac{\theta}{T}\left[\phi^{-1}\left(\frac{1}{T}\right)\right]^{-\frac1\theta}.
\end{align*}
Using~\eqref{sub-e62} and~\eqref{sub-e64} for the first integral, we obtain for any $p\in(0,1)$, $\theta\in(0,1/p)$ and $T>0$,
\begin{small}
\begin{equation}\label{sub-e70}
\begin{aligned}
    &(1-p\theta)\Gamma(1-p)\Ee\left[\left(\int_0^T t^{-\theta}\,\dif S_t\right)^p\right]\\
    &\leq T^{1-p\theta}\left[\phi^{-1}\left(\frac{1}{T}\right)\right]^{-p}
    \left[2^p\sum_{k=0}^\infty2^{pk}\phi\left(\frac{1}{2^k}\phi^{-1}\left(\frac{1}{T}\right)\right)
    +p\theta\sum_{k=0}^\infty 2^{-\frac{k}{\theta}}\phi\left(2^{k+1}\phi^{-1}\left(\frac{1}{T}\right)\right)\right].
\end{aligned}
\end{equation}\end{small}%
The conditions~\eqref{sub-e50} allow us (cf.\ Remark~\ref{sub-43}.\ref{sub-43-d}, \ref{sub-43-e}) to estimate the terms under the sum
for some $\tilde p > p$ and $\tilde\theta < 1/\theta$ for all large
values of $T$, say $T\geq 1$.
Therefore,
\begin{equation}\label{sub-e72}
\begin{aligned}
    (1-p\theta)&\Gamma(1-p)\Ee\left[\left(\int_0^T t^{-\theta}\,\dif S_t\right)^p\right]\\
    &\leq \left(2^p\sum_{k=0}^\infty2^{-(\tilde{p}-p)k} + p\theta2^{\tilde{\theta}}
    \sum_{k=0}^\infty2^{-\left(\frac{1}{\theta}-\tilde{\theta}\right)k}\right)
    T^{-p\theta}\left[\phi^{-1}\left(\frac{1}{T}\right)\right]^{-p},
\end{aligned}
\end{equation}
and~\ref{sub-41-iv} follows.

\medskip\noindent\ref{sub-41-v}
If we replace in the proof of~\ref{sub-41-iii} the conditions~\eqref{sub-e50} by~\eqref{sub-e52}, we get from from~\eqref{sub-e70} that~\eqref{sub-e72} holds for \emph{small} $T$, say $T\leq 1$,
and~\ref{sub-41-v} follows.

\medskip\noindent\ref{sub-41-vi}
Since $p<0$, we get by monotonicity
\begin{gather*}
    \Ee\left[\left(\int_0^T \eup^{-\lambda t}\,\dif S_t\right)^p\right]
    \leq
    \Ee\left[\left(\int_0^{T\wedge1} \eup^{-\lambda t}\,\dif S_t\right)^p\right]
    \leq
    \eup^{-p\lambda(T\wedge1)} \Ee \left[S_{T\wedge1}^p\right]
    \leq
    \eup^{-p\lambda}\Ee \left[S_{T\wedge1}^p\right].
\end{gather*}
Under the condition $\liminf_{s\to \infty} \phi(2s)/\phi(s)>1$ there is some constant $C_p$ such that
\begin{gather*}
    \Ee \left[S_{T\wedge1}^p\right]
    \leq C_p\left[\phi^{-1}\left(\frac{1}{T\wedge1}\right)\right]^{-p},
    \qquad T>0,
\end{gather*}
see~\cite[Theorem 2.1\,(ii)\,(c)]{DSS17}, and we get~\ref{sub-41-vi}.

\medskip\noindent\ref{sub-41-vii}
In view of Remark~\ref{sub-43}.\ref{sub-43-c} we may assume that $0<p<1$. We see with~\eqref{sub-e14}, Tonelli's theorem and Lemma~\ref{sub-21}
\begin{equation}\label{sub-e74}
\begin{aligned}
    \Gamma(1-p)\Ee\left[\left(\int_0^\infty \eup^{-\lambda t}\,\dif S_t\right)^p\right]
    &= p\Ee\left[\int_0^\infty\left(1-\exp\left[-r\int_0^\infty\eup^{-\lambda t}\,\dif S_t\right]\right)\frac{\dif r}{r^{p+1}}\right]\\
    &= p\int_0^\infty\left(1-\exp\left[-\int_0^\infty\phi\left(r\eup^{-\lambda t}\right) \dif t\right]\right)\frac{\dif r}{r^{p+1}}.
\end{aligned}
\end{equation}
Note the following elementary inequalities
\begin{gather}\label{sub-e78}
    \frac 12 (1\wedge x)
    \leq (1-\eup^{-1}) (1\wedge x)
    \leq 1-\eup^{-x}
    \leq 1\wedge x, \qquad x\geq 0.
\end{gather}
Assume that $\int_0^1 \phi(s) s^{-1-p}\,\dif s < \infty$. Using in~\eqref{sub-e74} the upper estimate from~\eqref{sub-e78} we get
\begin{align*}
    \Gamma(1-p)\Ee\left[\left(\int_0^\infty \eup^{-\lambda t}\,\dif S_t\right)^p\right]
    &\leq p\int_0^1\left(\int_0^\infty\phi\left(r\eup^{-\lambda t}\right) \dif t\right)\frac{\dif r}{r^{p+1}}+p\int_1^\infty\frac{\dif r}{r^{p+1}}\\
    &= \frac{p}{\lambda}\int_0^1\left(\int_s^1\frac{\dif r}{r^{p+1}}\right)\frac{\phi(s)}{s}\,\dif s+1\\
    &\leq\frac{p}{\lambda}\int_0^1\left(\int_s^\infty\frac{\dif r}{r^{p+1}}\right)\frac{\phi(s)}{s}\,\dif s+1\\
    &=\frac{1}{\lambda}\int_0^1\frac{\phi(s)}{s^{p+1}}\,\dif s +1.
\end{align*}
This proves the direction ``$\Rightarrow$'' in~\eqref{sub-e54}.

In order to see the other implication, we assume that $\Ee\left[\left(\int_0^\infty \eup^{-\lambda t}\,\dif S_t\right)^p\right]<\infty$. Because of Proposition~\ref{sub-25} (applied with $f(t) = \eup^{-\lambda t}$) this means that $\int_0^1 \phi(s) s^{-1}\,\dif s=\lambda\int_0^\infty\phi\left(\eup^{-\lambda t}\right)\,\dif t< \infty$. Applying the lower estimate from~\eqref{sub-e78} to~\eqref{sub-e74} gives
\begin{align*}
    \Gamma(1-p)\Ee\left[\left(\int_0^\infty \eup^{-\lambda t}\,\dif S_t\right)^p\right]
    &\geq \frac p2 \int_0^\infty 1\wedge \left(\int_0^\infty\phi\left(r\eup^{-\lambda t}\right) \dif t\right) \frac{\dif r}{r^{p+1}}\\
    &= \frac p2 \int_0^\infty 1\wedge \left(\frac 1\lambda \int_0^r \frac{\phi\left(s\right)}{s}\,\dif s\right) \frac{\dif r}{r^{p+1}}\\
    &\geq \frac p{2\lambda} \int_0^{r_0}\left(\int_0^r \frac{\phi\left(s\right)}{s}\,\dif s\right)\, \frac{\dif r}{r^{p+1}}
\intertext{where $r_0=r_0(\lambda)>0$ is so small that $\int_0^{r_0} \phi\left(s\right)s^{-1}\,\dif s \leq \lambda$. Thus, by Tonelli's theorem,}
    \Gamma(1-p)\Ee\left[\left(\int_0^\infty \eup^{-\lambda t}\,\dif S_t\right)^p\right]
    &\geq \frac p{2\lambda} \int_0^{r_0}\left(\int_s^{r_0}\frac{\dif r}{r^{p+1}}\right)\frac{\phi\left(s\right)}{s}\,\dif s\\
    &= \frac 1{2\lambda} \int_0^{r_0} \frac{\phi\left(s\right)}{s^{p+1}}\,\dif s
    - \frac 1{2\lambda r_0^p} \int_0^{r_0} \frac{\phi\left(s\right)}{s}\,\dif s.
\end{align*}
Since the second summand is finite, the proof is complete.
\end{proof}

\section{Stochastic convolutions}\label{con}
Recall that $\mathds{W} = (W_t)_{t\geq 0}$ is a cylindrical Brownian motion with values in a Hilbert space $H$, and that $\mathds{S} = (S_t)_{t\geq 0}$ is a subordinator which is independent of $\mathds{W}$. We assume that $t\mapsto S_t$ is a.s.\ strictly increasing; this is equivalent to assuming that $\lim_{\xi\to\infty}\phi(\xi)=\infty$. In order to justify the method of conditioning on the subordinator $\mathds{S}$, we introduce in this section a product construction which will allow us to freeze the subordinator, see e.g.  \cite{Zhang13,WaXuZh15}. 

Let $\Omega^\mathds{W}$ be the space of all continuous functions from $\omega: [0,\infty) \to H$, $t\mapsto\omega_t$, which vanish at $t=0$; we endow $\Omega^{\mathds{W}}$ with the topology of locally uniform convergence and the Wiener measure $\Pp^{\mathds{W}}$; under $\Pp^{\mathds{W}}$, the canonical process $(\omega_t)_{t \ge 0}$ is a cylindrical Brownian motion valued on $H$,  that is
\begin{gather*}
    W_t(\omega)=\omega_t, \quad t \geq 0.
\end{gather*}
Similarly, we construct a canonical realization of the subordinator $(S_t)_{t\geq 0}$ on the space $\Omega^{\mathds{S}}$ of all strictly increasing c\`adl\`ag functions $\ell : [0,\infty) \to [0,\infty)$, $t\mapsto\ell_t$, such that $\ell_0=0$; we endow $\mathds{S}$ with the Skorohod topology and an probability measure $\Pp^{\mathds{S}}$, such that $(S_t)_{t\geq 0}$ is the canonical coordinate process
\begin{gather*}
    S_t(\ell)=\ell_t, \quad t \geq 0.
\end{gather*}
We consider the SPDE \eqref{int-e02} on the product space
\begin{gather*}
    (\Omega, \mathcal F, \Pp)
    := (\Omega^{\mathds{W}} \times \Omega^{\mathds{S}}, \mathcal B(\Omega^{\mathds{W}} \times \Omega^{\mathds{S}}), \Pp^{\mathds{W}}\otimes \Pp^{\mathds{S}}).
\end{gather*}

Recall the unique solution $(X_t)_{t \ge 0}$ of the SPDE \eqref{int-e02} is defined as \eqref{int-e10}, and the stochastic convolution $(Z_t)_{t \ge 0}$ is defined as \eqref{con-e02}. For a $\ell \in \mathds{S}$, we have $S_t(\ell)=\ell_t$ for $t \ge 0$, i.e., the subordinator takes a sample path $(\ell_{t})_{t \ge 0}$,  let us consider the following SDE in the probability space
$(\Omega^{\mathds{W}}, \mathcal B(\Omega^{\mathds{W}}), \Pp^{\mathds{W}})$:
\begin{equation} \label{int-e02-1}
    \dif X^{\ell}_t=[-AX^{\ell}_t+F(X^{\ell}_t)]\,\dif t+Q(X^{\ell}_{t-})\,\dif W_{\ell_t}, \quad X^\ell_0=x.
\end{equation}
Note that $(W_{\ell_t})_{t\geq 0}$ is a c\`adl\`ag martingale (with deterministic jump-times) for the filtration $\mathscr{G}_{\ell_t}$ where $(\mathscr{G}_t)_{t\geq 0}$ is the filtration of the cylindrical Brownian motion $\mathbb{W}$. Identifying cylindrical Brownian motion with a Hilbert space valued Wiener process (in general, a larger Hilbert space $U\supset H$, cf.\ \cite[Theorem 7.13]{PeZa07}) we can use the results in \cite[\S 8.1]{PeZa07} to see that the bracket of the time-changed process satisfies $\langle W_{\ell},W_{\ell}\rangle_t = \langle W,W\rangle_{\ell_t}$, and the same relation holds for the operator (or tensor) bracket.

{By \cite[p.142]{PeZa07}}, the SPDE \eqref{int-e02-1} has a unique mild solution given by
\begin{gather*}
    X_t^\ell(x) = e^{-tA}x + \int_0^t e^{-(t-s)A}F(X_s^\ell(x))\,\dif s + Z^{\ell}_t
\end{gather*}
where $Z_t^{\ell}$ is the \normal following stochastic convolution:
\begin{equation}\label{con-e04}
    Z^{\ell}_t
    = \int_0^t \eup^{-(t-s)A} Q(X^{\ell}_{s-}(x))\,\dif W_{\ell_s}.
\end{equation}

To keep notation simple, we will suppress the initial condition $x$ and write $X^{\ell}_{s-} = X^{\ell}_{s-}(x)$. Moreover, we use $X=(X_t)_{t \ge 0}$, $Z=(Z_t)_{t \ge 0}$, $X^\ell=(X^\ell_t)_{t \ge 0}$ and $Z^\ell=(Z^\ell_t)_{t \ge 0}$.\normal
\begin{lemma} \label{l:DPWPS}
Let $F$ and $G$ be measurable functionals of $X$ and $Z$ respectively. We have
\begin{align*}
    \Ee[F(X)] &=\Ee^{{\mathds{S}}}\left[\Ee^{{\mathds{W}}}(F(X^\ell))|_{\ell=S}\right],\\
    \Ee[G(Z)] &=\Ee^{{\mathds{S}}}\left[\Ee^{{\mathds{W}}}(F(Z^\ell))|_{\ell=S}\right].
\end{align*}
\end{lemma}

\begin{proof}
Because of \eqref{int-e10}, $X$ is a measurable map of $x$ and  $W_{S}$ which we will denote by $X:=\bar X(x,W_{S})$. We have
\begin{align*}
\Ee[F(X)]&=\Ee^{\Pp^{\mathds{W}} \times \Pp^{\mathds{S}}} \left[F(\bar X(x,W_{S}))\right] \\
&=\int_{\mathds{S}} \Ee^{{\mathds{W}}} \left[F(\bar X(x,W_{\ell}))\right]  \Pp^{\mathds{S}}(\dif \ell) \\
&=\int_{\mathds{S}} \Ee^{{\mathds{W}}} \left[F(X^\ell(x))\right]  \Pp^{\mathds{S}}(\dif \ell)
\end{align*}
where the last equality follows from the fact that $\bar X(x,W_{\ell})$ is the solution to Equation~\eqref{int-e02-1}.

It is easy to see that $Z=(Z_t)_{t \ge 0}$ is a measurable functional of $X$ and $W_S$; we denote it by $Z=\bar Z(X,W_S)$, so for any measurable function $G$ of $Z$, we have
\begin{align*}
\Ee\left[G(Z)\right]&=\Ee^{\Pp^{\mathds{W}} \times \Pp^{\mathds{S}}}\left[G(\bar Z(X,W_S)\right] \\
&=\Ee^{\Pp^{\mathds{W}} \times \Pp^{\mathds{S}}}\left[G(\bar Z(\bar X(x,W_S),W_S)\right] \\
&=\int_{\mathds{S}} \Ee^{{\mathds{W}}}  \left[G(\bar Z(\bar X(x,W_{\ell}), W_\ell)\right]  \Pp^{\mathds{S}}(\dif \ell) \\
&=\int_{\mathds{S}} \Ee^{{\mathds{W}}}  \left[G(\bar Z(X^\ell(x), W_\ell)\right]  \Pp^{\mathds{S}}(\dif \ell) \\
&=\int_{\mathds{S}} \Ee^{{\mathds{W}}}  \left[G(Z^\ell)\right]  \Pp^{\mathds{S}}(\dif \ell)
\end{align*}
where the last inequality is due to $Z^\ell=\bar Z(X^\ell(x), W_\ell)$.
\end{proof}


\subsection{Estimate of the $p$-th moment of $Z_t$}
From Lemma \ref{l:DPWPS}, in order to estimate $X_t$ and $Z_t$, we can first estimate $X^\ell_t$ and $Z^\ell_t$ for a given $\ell \in \mathds{S}$ and then integrate the estimations over $\mathds{S}$.
\begin{theorem}\label{con-11}
    Let $(W_t)_{t\geq 0}$, $(S_t)_{t\geq 0}$ and $(Z_t)_{t\geq 0}$ be as above and assume \eqref{A1} and \eqref{A2}.
\begin{enumerate}
\item\label{con-11-i}
    Assume that $\inf_{s>0} \phi(2s)/\phi(s) > 1$ and let $p,\theta>0$ be such that
    \begin{gather*}
        \frac{p}{2}
        < \log_2\left(\inf_{s>0}\frac{\phi(2s)}{\phi(s)}\right)
        \leq \log_2\left(\limsup_{s\to\infty}\frac{\phi(2s)}{\phi(s)}\right)
        < \frac{1}{2\theta}.
    \end{gather*}
    There exists a constant $C=C(p,\theta)>0$ such that
	\begin{gather*}
        \Ee \left[|A^{\theta}Z_t|^p\right]
        \leq Ct^{-p\theta}\left[\phi^{-1}\left(\frac{1}{t}\right)\right]^{-\frac{p}{2}}
        \quad \text{for all $t\in(0,1]$}.
	\end{gather*}

\item\label{con-11-ii}
    Assume that $\liminf_{s\to 0}\phi(2s)/\phi(s) > 1$ and let $p,\theta>0$ be such that
    \begin{equation}\label{con-e12}
        \frac{p}{2}
        < \log_2\left(\liminf_{s\to 0} \frac{\phi(2s)}{\phi(s)}\right)
        \leq \log_2\left(\sup_{s>0} \frac{\phi(2s)}{\phi(s)}\right)
        < \frac{1}{2\theta},
    \end{equation}
    then one has
	\begin{gather*}
	   \sup_{t>0} \Ee \left[|A^{\theta}Z_t|^p\right] < \infty.
	\end{gather*}
\end{enumerate}
\end{theorem}
\begin{proof}
Since $\phi(2s)\leq 2\phi(s)$, our assumptions guarantee that $p\leq 2$. An application of Jensen's inequality and It\^o's isometry (e.g.\ \cite[Theorem 8.7]{PeZa07}) \normal  shows
\begin{align*}
    \Ee^{\mathds{W}} \left[\left|A^{\theta} Z^{\ell}_{t}\right|^{p}\right]
    &= \Ee^{\mathds{W}} \left[\left|\int_{0}^{t} A^{\theta} \eup^{-(t-s)A} Q(X^{\ell}_{s-})\,\dif W_{\ell_{s}} \right|^{p}\right]\\
    &\leq \left(\Ee^{\mathds{W}} \left[\left|\int_{0}^{t} A^{\theta} \eup^{-(t-s)A} Q(X^{\ell}_{s-})\,\dif W_{\ell_{s}} \right|^{2}\right]\right)^{\frac{p}{2}}\\
    &\leq \left(\int_{0}^{t} \|A^{\theta}\eup^{-(t-s)A}\|^{2}\cdot \|Q\|^{2}_{\mathrm{HS},\infty}\,\dif \ell_{s} \right)^{\frac{p}{2}}.
\intertext{By \eqref{int-e20} and \eqref{int-e24} we get}
    \Ee^{\mathds{W}} \left[\left|A^{\theta} Z^{\ell}_{t}\right|^{p}\right]
    &\leq \|Q\|^{p}_{\mathrm{HS},\infty}\left(\int_{0}^{t} \|A^{\theta} \eup^{-\frac 12 A(t-s)}\|^{2} \|\eup^{-\frac 12 A(t-s)}\|^{2}\,\dif \ell_{s} \right)^{\frac{p}{2}}\\
    &\leq C_\theta  \|Q\|^{p}_{\mathrm{HS},\infty}\left(\int_{0}^{t} (t-s)^{-2\theta}\eup^{-\gamma_1(t-s)} \,\dif \ell_{s} \right)^{\frac{p}{2}}.
\end{align*}
This estimate, together with Lemma \ref{sub-23}, yields
\begin{align*}
    \Ee \left[\left|A^{\theta} Z_{t}\right|^{p}\right]
    &= \Ee^{\mathds{S}} \left(\Ee^{\mathds{W}} \left[\left|A^{\theta} Z^{\ell}_{t}\right|^{p}\right]\Big|_{\ell=S}\right)\\
    &\leq C_\theta  \|Q\|^{p}_{\mathrm{HS},\infty} \Ee \left[\left(\int_{0}^{t} (t-r)^{-2\theta}\eup^{-\gamma_1(t-r)}\,\dif S_{r} \right)^{\frac{p}{2}}\right]\\
	&= C_\theta \|Q\|^{p}_{\mathrm{HS},\infty}\Ee \left[\left(\int_{0}^{t} r^{-2\theta}\eup^{-\gamma_1 r}\,\dif S_{r} \right)^{\frac{p}{2}}\right].
\end{align*}
The assertions~\ref{con-11-i} and~\ref{con-11-ii} follow directly from this estimate:

\medskip\noindent
\ref{con-11-i} Since we have
\begin{gather*}
    \Ee \left[\left(\int_{0}^{t} r^{-2\theta}\eup^{-\gamma_1 r}\,\dif S_{r} \right)^{\frac{p}{2}}\right]
    \leq\Ee \left[\left(\int_{0}^{t} r^{-2\theta}\,\dif S_{r} \right)^{\frac{p}{2}}\right],
\end{gather*}
we get~\ref{con-11-i} directly from Proposition~\ref{sub-41}.\ref{sub-41-v}.

\medskip\noindent
\ref{con-11-ii} Observe that
\begin{align*}
    \sup_{t>0}\left(\int_{0}^{t} r^{-2\theta}\eup^{-\gamma_1 r}\,\dif S_{r} \right)^{\frac{p}{2}}
    &\leq \left(\int_{0}^{1} r^{-2\theta}\,\dif S_{r} + \int_{1}^\infty\eup^{-\gamma_1 r}\,\dif S_{r}\right)^{\frac{p}{2}}\\
    &\leq \left(\int_{0}^{1} r^{-2\theta}\,\dif S_{r}\right)^{\frac{p}{2}} + \left(\int_{0}^\infty\eup^{-\gamma_1 r}\,\dif S_{r}\right)^{\frac{p}{2}}.
\end{align*}
From Proposition \ref{sub-41}.\ref{sub-41-iv} we know that
$
    \Ee\left[\left(\int_{0}^{1} r^{-2\theta}\,\dif S_{r}\right)^{{p}/{2}}\right]
    < \infty
$.
On the other hand, $\frac p2 < \log_2\left(\liminf_{s\to 0} {\phi(2s)}/{\phi(s)}\right)$ allows us to use Proposition~\ref{sub-41}.\ref{sub-41-vii}
with $p$ replaced by $p/2$ (see Lemma~\ref{con-33} below), and so
$
    \Ee\left[\left(\int_{0}^\infty\eup^{-\gamma_1 r}\,\dif S_{r}\right)^{{p}/{2}}\right]
    < \infty.
$
This completes the proof.
\end{proof}

\subsection{A maximal inequality and a small ball probability of $Z_t$}
\begin{theorem}\label{con-21}
    Let $(W_t)_{t\geq 0}$, $(S_t)_{t\geq 0}$ and $(Z_t)_{t\geq 0}$ be as above and assume \eqref{A1}, \eqref{A2}. If
    \begin{gather*}
        0<p<2\log_2\left(\liminf_{s\to 0} \frac{\phi(2s)}{\phi(s)}\right),
        \quad\text{resp.,}\quad
        0<p<2\log_2\left(\inf_{s>0} \frac{\phi(2s)}{\phi(s)}\right),
    \end{gather*}
    then there exists a constant $C=C(p,\|Q\|_{\mathrm{HS},\infty})>0$ such that
    \begin{equation}\label{con-e22}
        \Ee \left[\sup_{0 \leq t \leq T} |Z_t|^p\right]
        \leq C \left[\phi^{-1}\left(\frac{1}{T}\right)\right]^{-\frac{p}{2}}
        \quad\text{holds for all $T\geq 1$, resp., $T>0$}.
    \end{equation}
\end{theorem}

Let us note an immediate consequence of Theorem \ref{con-21}.
\begin{corollary}\label{con-23}
    Assume that \eqref{A1}, \eqref{A2} and $\liminf_{s\to 0}\phi(2s)/\phi(s) > 1$ hold. Then $Z. \in L^\infty([0,T], H)$ a.s..
\end{corollary}

A further consequence of the maximal inequality is the following small ball probability estimate.
\begin{theorem}\label{con-25}
    Let $(W_t)_{t\geq 0}$, $(S_t)_{t\geq 0}$ and $(Z_t)_{t\geq 0}$ be as above, and assume \eqref{A1}, \eqref{A2}.
    \begin{enumerate}
    \item\label{con-25-i}
        If $\phi$ has zero drift, then for any $\delta\in(0,1)$ and any $T>0$, the following small ball estimate holds
    \begin{gather*}
        \Pp\left(\sup_{0 \leq t \leq T} |Z_t|<\delta\right)>0.
    \end{gather*}

    \item\label{con-25-ii}
    If $\inf_{s>0}\frac{\phi(2s)}{\phi(s)}>1$, then for any $\delta\in(0,1)$, any $\kappa\in(0,1)$,
    and any $0<p<\log_2\left(\inf_{s>0}\frac{\phi(2s)}{\phi(s)}
    \right)$, there exists some $C=C(p,\kappa,\|Q\|_{\mathrm{HS},\infty})>0$
    such that for sufficiently small $T>0$
    \begin{gather*}
        \Pp\left(\sup_{0 \leq t \leq T} |Z_t|<\delta\right)
        \geq\kappa\left(
        1-C\left[\delta^4
        \phi^{-1}\left(\frac{1}{T}\right)\right]^{-p}
        \right)>0.
    \end{gather*}
\end{enumerate}
\end{theorem}

The proofs of Theorems \ref{con-21} and \ref{con-25} rely on the following auxiliary result.
\begin{lemma} \label{con-27}
    Let $(W_t)_{t\geq 0}$, $(S_t)_{t\geq 0}$ and $(Z_t^\ell)_{t\geq 0}$ be as above and assume \eqref{A1}, \eqref{A2}.
    The following maximal inequality holds:
    \begin{gather*}
        \Ee^{\mathds{W}}\left[\sup_{0 \leq t \leq T} |Z^\ell_t|^2\right]
        \leq 9\|Q\|^2_{\mathrm{HS},\infty} \ell_T.
    \end{gather*}
\end{lemma}
\begin{proof}
    Note that $\dif Z_t^\ell = Q(X_{t-}^\ell)\,\dif W_{\ell_t} - AZ_{t}^\ell\,\dif t$. By It\^{o}'s formula \cite[Theorem 27.2]{Met82}, we have
    \begin{align*}
        &|Z^\ell_t|^2+2 \int_0^t |A^{\frac 12} Z^\ell_s|^2\,\dif s\\
        &\qquad=2 \int_0^t \langle Z^\ell_{s-}, Q(X^{\ell}_{s-})\,\dif W_{\ell_s}\rangle
        +\sum_{0 < s\leq t} \left[|Z^\ell_s|^2-|Z^\ell_{s-}|^2-2 \langle Z^\ell_{s-}, Q(Z^\ell_{s-})\Delta W_{\ell_s}\rangle \right].
    \end{align*}
    A direct calculation shows that
    \begin{align*}
        |Z^\ell_s|^2 &-|Z^\ell_{s-}|^2-2 \langle Z^\ell_{s-}, Q(Z^\ell_{s-})\Delta W_{\ell_s} \rangle\\
        &=|Z^\ell_{s-}+ Q(Z^\ell_{s-})\Delta W_{\ell_s}|^2-|Z^\ell_{s-}|^2-2 \langle Z^\ell_{s-}, Q(Z^\ell_{s-})\Delta W_{\ell_s}\rangle\\
        &=|Q(X^\ell_{s-}) \Delta W_{\ell_s}|^2,
    \end{align*}
    and, therefore,
    \begin{equation} \label{con-e28}
        |Z^\ell_t|^2+2 \int_0^t |A^{\frac 12} Z^\ell_s|^2\,\dif s
        = 2 \int_0^t \langle Z^\ell_{s-}, Q(X^{\ell}_{s-})\,\dif W_{\ell_s}\rangle
        + \sum_{0 < s \leq t} |Q(X^\ell_{s-}) \Delta W_{\ell_s}|^2.
    \end{equation}
    If we take expectations on both sides of the above equality and apply It\^o's isometry, we see
    \begin{align*}
        \Ee^{\mathds{W}} \left[|Z^\ell_t|^2\right] + 2 \int_0^t \Ee^{\mathds{W}} \left[|A^{\frac 12} Z^\ell_s|^2\right]\dif s
        &= \sum_{0 < s \leq t} \Ee^{\mathds{W}} \left[|Q(X^\ell_{s-}) \Delta W_{\ell_s}|^2\right]\\
        &= \int_0^t \Ee^{\mathds{W}} \left[\|Q(X^{\ell}_{s-})\|^2_\mathrm{HS}\right]\dif \ell_s.
    \end{align*}
    This implies, in particular,
    \begin{equation}\label{con-e30}
    \Ee^{\mathds{W}} \left[|Z^\ell_t|^2\right]
    \leq \int_0^t \Ee^{\mathds{W}} \left[\|Q(X^{\ell}_{s-})\|^2_\mathrm{HS}\right]\dif \ell_s
    \leq \|Q\|^2_{\mathrm{HS},\infty} \ell_t.
    \end{equation}
    Since $Z^\ell_t-Z^\ell_{t-} = Q(X^\ell_{t-}) \Delta W_{\ell_t}$, we see using \eqref{con-e30} and, again, It\^o's isometry,
    \begin{equation}\label{con-e32}
    \begin{aligned}
        \Ee^{\mathds{W}} \left[|Z^\ell_{t-}|^2\right]
        &=\Ee^{\mathds{W}} \left[\left|Z^\ell_t-Q(X^\ell_{t-}) \Delta W_{\ell_t}\right|^2\right]\\
        &\leq 2 \Ee^{\mathds{W}} \left[\left|Z^\ell_t\right|^2\right] + 2 \Ee^{\mathds{W}} \left[\|Q(X^{\ell}_{s-})\|^2_\mathrm{HS}\right] \Delta {\ell_t}\\
        &\leq 2 \|Q\|^2_{\mathrm{HS},\infty} \ell_t + 2 \|Q\|^2_{\mathrm{HS},\infty}  \Delta \ell_t\\
        &\leq 4 \|Q\|^2_{\mathrm{HS},\infty} \ell_t.
    \end{aligned}
    \end{equation}

    Combining \eqref{con-e28}, the Cauchy-Schwarz inequality, Doob's martingale maximal $L^2$-inequality and It\^{o}'s isometry, we get
    \begin{align*}
        \Ee^{\mathds{W}} &\left[\sup_{0 \leq t \leq T} |Z^\ell_t|^2\right]\\
        &\leq 2 \Ee^{\mathds{W}}\left[\sup_{0 \leq t \leq T} \left|\int_0^t \langle Z^\ell_{s-}, Q(X^{\ell}_{s-})\,\dif W_{\ell_s}\rangle\right|\right]
        + \Ee^{\mathds{W}}\left[\sum_{0 < s \leq T} |Q(X^\ell_{s-}) \Delta W_{\ell_s}|^2\right] \\
         &\leq 2 \left\{\Ee^{\mathds{W}}\left[\sup_{0 \leq t \leq T} \left|\int_0^t \langle Z^\ell_{s-}, Q(X^{\ell}_{s-})\,\dif W_{\ell_s}\rangle\right|^2\right]\right\}^{\frac12}
        + \Ee^{\mathds{W}}\left[\sum_{0 < s \leq T} |Q(X^\ell_{s-}) \Delta W_{\ell_s}|^2\right] \\
        &\leq 4 \left(\Ee^{\mathds{W}} \left[\left|\int_0^T \langle Z^\ell_{s-}, Q(X^{\ell}_{s-}) \,\dif W_{\ell_s}\rangle\right|^2\right]\right)^{\frac 12}
        + \int_0^T \|Q(X^\ell_{s-})\|_\mathrm{HS}^2\,\dif \ell_s\\
        &\leq 4 \|Q\|_{\mathrm{HS}, \infty}  \left(\int_0^T \Ee^{\mathds{W}} \left[|Z^\ell_{s-}|^2\right]\dif \ell_s\right)^{\frac 12}
        + \|Q\|_{\mathrm{HS}, \infty}^2 \ell_T.
    \end{align*}
    This estimate together with \eqref{con-e32} finally yields
    \begin{align*}
        \Ee^{\mathds{W}} \left[\sup_{0 \leq t \leq T} |Z^\ell_t|^2\right]
        &\leq 4 \|Q\|_{\mathrm{HS}, \infty}  \left(\int_0^T  4 \|Q\|^2_{\mathrm{HS},\infty} \ell_t\,\dif \ell_t\right)^{\frac 12}
        + \|Q\|_{\mathrm{HS}, \infty}^2 \ell_T\\
        &\leq 8  \|Q\|^2_{\mathrm{HS}, \infty} \left(\int_0^T  \ell_T\,\dif \ell_t\right)^{\frac 12}
        + \|Q\|_{\mathrm{HS}, \infty}^2 \ell_T  \\
        &=9 \|Q\|_{\mathrm{HS}, \infty}^2 \ell_T.
    \qedhere
    \end{align*}
\end{proof}

\begin{proof} [Proof of Theorem \ref{con-21}]
    Since $\mathds{W}$ and $\mathds{S}$ are independent, we get with Jensen's inequality
    \begin{align*}
        \Ee\left[\sup_{0 \leq t \leq T} |Z_t|^p\right]
        &= \Ee^{\mathds{S}} \left(\Ee^{\mathds{W}} \left[\sup_{0 \leq t \leq T} |Z^\ell_t|^p\right]\bigg|_{\ell=S}\right)\\
    &\leq \Ee^{\mathds{S}} \left(\Ee^{\mathds{W}} \left[\sup_{0 \leq t \leq T} |Z^\ell_t|^2\right]\bigg|_{\ell=S}\right)^{\frac{p}{2}}\\
    &\leq 3^p \|Q\|_{\mathrm{HS}, \infty}^p \Ee^{\mathds{S}} \left[S_T^{{p}/{2}}\right].
    \end{align*}
    The claim follows now from Proposition \ref{sub-41}.\ref{sub-41-iii}.
\end{proof}

\begin{proof}[Proof of Theorem \ref{con-25}]
By Chebyshev's inequality and Lemma \ref{con-27}, we have
for all $T>0$
\begin{align*}
    \Pp\left(\sup_{0 \leq t \leq T} |Z_t| \geq  \delta,\; S_T<\delta^4\right)
    &=\Ee^{\mathds{S}} \left(\left[\I_{\{\ell_T<\delta^4\}}
    \Pp^{\mathds{W}} \left(\sup_{0 \leq t \leq T} |Z^\ell_t|  \geq  \delta\right)\right]_{\ell=S}\right)\\
    &\leq \delta^{-2}\Ee^{\mathds{S}}\left(\left[\I_{\{\ell_T<\delta^4\}}
    \Ee^{\mathds{W}}\left[\sup_{0 \leq t \leq T} |Z^\ell_t|^2\right]\right]_{\ell=S}\right)\\
    &\leq 9\|Q\|_{\mathrm{HS}, \infty}^2\delta^{-2} \Ee^{\mathds{S}}\left(\I_{\{S_T<\delta^4\}} S_T\right)\\
    &\leq 9\|Q\|_{\mathrm{HS}, \infty}^2\delta^2 \Pp\left(S_T<\delta^4\right).
\end{align*}
Let $\kappa\in(0,1)$.
If $0<\delta<\sqrt{1-\kappa}/(3\|Q\|_{\mathrm{HS},\infty})$,
we get
\begin{equation}\label{lowerbound}
\begin{aligned}
    \Pp\left(\sup_{0 \leq t \leq T} |Z_t| < \delta\right)
    &\geq
    \Pp\left(\sup_{0 \leq t \leq T} |Z_t| < \delta,\; S_T<\delta^4\right)\\
    &= \Pp\left(S_T<\delta^4\right)-\Pp\left(\sup_{0 \leq t \leq T} |Z_t|  \geq \delta,\; S_T<\delta^4\right)\\
    &\geq \left(1-9\|Q\|_{\mathrm{HS}, \infty}^2\delta^2\right)
    \Pp\left(S_T<\delta^4\right)\\
    &\geq \kappa \Pp\left(S_T<\delta^4\right).
\end{aligned}
\end{equation}

\medskip\noindent
\ref{con-25-i} If $\phi$ has zero drift, then $\Pp\left(S_T<\delta^4\right)>0$ for any
$\delta>0$ and $T>0$,  see e.g.\ \cite[Corollary 24.8]{sato} and observe that the support of $S_T$ contains zero. This, with \eqref{lowerbound}, gives the first assertion.

\medskip\noindent
\ref{con-25-ii}
By Chebyshev's inequality and Proposition~\ref{sub-41}.\ref{sub-41-iii}, there exists
$C_1=C_1(p)>0$ such that
\begin{align*}
    \Pp\left(S_T<\delta^4\right)
    &=1-\Pp\left(S_T^p\geq\delta^{4p}\right)\\
    &\geq1-\delta^{-4p}\Ee\left[S_T^p\right]\\
    &\geq1-C_1
    \left[\delta^4\phi^{-1}\left(\frac{1}{T}\right)\right]^{-p}.
\end{align*}
Combining this with \eqref{lowerbound}, we get for all $T>0$ and $0<\delta<\sqrt{1-\kappa}/(3\|Q\|_{\mathrm{HS},\infty})$
\begin{gather*}
    \Pp\left(\sup_{0 \leq t \leq T} |Z_t| < \delta\right)
    \geq\kappa\left(
    1-C_1
    \left[\delta^4\phi^{-1}\left(\frac{1}{T}\right)\right]^{-p}
    \right).
\end{gather*}
Fix $\delta\in(0,1)$ and pick $C_2=C_2(\kappa,\|Q\|_{\mathrm{HS},\infty})>1$ such that
\begin{gather*}
    \tilde{\delta}:=\frac{\delta}{C_2}\leq
    \delta\wedge
    \frac{\sqrt{1-\kappa}}{3\|Q\|_{\mathrm{HS},\infty}}.
\end{gather*}
Then, for all $\delta\in(0,1)$ and $T>0$
\begin{align*}
    \Pp\left(\sup_{0 \leq t \leq T} |Z_t| < \delta\right)
    &\geq
    \Pp\left(\sup_{0 \leq t \leq T} |Z_t| < \tilde{\delta}\right)\\
    &\geq\kappa\left(
    1-C_1
    \left[\tilde{\delta}^4\phi^{-1}\left(\frac{1}{T}\right)
    \right]^{-p}
    \right)\\
    &=\kappa\left(
    1-C_1C_2^{4p}
    \left[\delta^4\phi^{-1}\left(\frac{1}{T}\right)
    \right]^{-p}
    \right).
\end{align*}
Since $\inf_{s>0}\frac{\phi(2s)}{\phi(s)}>1$ implies
$\lim_{s\rightarrow\infty}\phi(s)=\infty$, we know that
the last term is positive if $T>0$ is small enough.
This completes the proof.
\end{proof}

In the proof of Theorem~\ref{con-11}, we need the following auxiliary result.
\begin{lemma}\label{con-33}
    Let $g:(0,1)\to(0,\infty)$ be an increasing function. Then
    \begin{gather*}
        \kappa := \log_2\left(\liminf_{s\to 0}\frac{g(2s)}{g(s)}\right) > 0
    \implies
        \forall\epsilon>0 \::\: \limsup_{s\to 0}\frac{g(s)}{s^{\kappa-\epsilon}} < \infty.
    \end{gather*}
\end{lemma}
\begin{proof}
    It follows from the assumption that for any $\epsilon \in (0,\kappa)$ there exists some sufficiently small $\delta=\delta(\epsilon)>0$ such that
    \begin{gather*}
        2^{\kappa-\epsilon}g(s) \leq g(2s),
        \quad  0<s\leq\delta.
    \end{gather*}
    By iteration, we get for any $n\in\nat$,
    \begin{gather*}
        g(s)\leq 2^{-n(\kappa-\epsilon)}
        g\left(2^{n}s\right),
        \quad 0<s\leq \delta 2^{-n+1}.
    \end{gather*}
    For $s\in(0,\delta]$ there is a unique $n=n_s\in\nat$ such that $2^{-n}\delta < s \leq 2^{-n+1}\delta$. Since $g$ is increasing,
    \begin{gather*}
        g(s)
        \leq g\left(\delta 2^{-n+1}\right)
        \leq 2^{-n(\kappa-\epsilon)} g\left(2^n \delta 2^{1-n}\right)
        \leq\left(\frac{s}{\delta}\right)^{\kappa-\epsilon}
        g(2\delta)
        =\frac{g(2\delta)}{\delta^{\kappa-\epsilon}}\,
        s^{\kappa-\epsilon},
    \end{gather*}
    which implies that
    \begin{gather*}
        \limsup_{s\to 0}\frac{g(s)}{s^{\kappa-\epsilon}}
        \leq \frac{g(2\delta)}{\delta^{\kappa-\epsilon}}
        < \infty.
    \qedhere
    \end{gather*}
\end{proof}

\section{Invariant measures}\label{inv}
Recall that the solution to Eq.~\eqref{int-e02} has the following form:
\begin{gather*}
    X_t=\eup^{-tA} x+\int_0^t\eup^{-(t-s)A} F(X_s)\,\dif s+Z_t,
\end{gather*}
where
\begin{gather*}
    Z_t=\int_0^t\eup^{-(t-s)A} Q(X_{s-} )\,\dif W_{S_s}.
\end{gather*}

\begin{theorem} \label{inv-11}
Assume \eqref{A1}--\eqref{A3}, and $\liminf_{s\to 0}\phi(2s)/\phi(s) > 1$.
Then the system \eqref{int-e02} admits at least one invariant measure.
\end{theorem}

\begin{proof}
Pick $p>0$ and $\theta\in(0,1)$ such that \eqref{con-e12} holds. Note that
\begin{gather*}
    \Ee \left[|A^\theta X_t|^p\right]
    \leq
    3^{p}\left(|A^\theta\eup^{-tA} x|^p+\left[\int_0^t \|A^\theta\eup^{-(t-s)A}\|\,\dif s\right]^p \|F\|^p_\infty
    + \Ee \left[|A^\theta Z_t|^p\right] \right).
\end{gather*}
By \eqref{int-e20}, we have for $0\leq s\leq t$
\begin{gather*}
    \|A^\theta\eup^{-(t-s)A}\|
    \leq \|A^\theta\eup^{-\frac 12 A(t-s)}\|\|\eup^{-\frac 12A(t-s)}\|
    \leq C_\theta (t-s)^{-\theta} \eup^{-\gamma_1(t-s)/2}
\end{gather*}
which implies that
\begin{align*}
    \sup_{t>0}\int_0^t \|A^\theta \eup^{-(t-s)A}\|\,\dif s
    &\leq C_\theta\sup_{t>0}\int_0^t (t-s)^{-\theta} \eup^{-\gamma_1(t-s)/2}\,\dif s\\
    &= C_\theta\int_0^\infty s^{-\theta} \eup^{-\gamma_1s/2}\,\dif s\\
    &=: C_{\theta,\gamma_1}.
\end{align*}
Then by Theorem \ref{con-11}.\ref{con-11-ii},
\begin{gather*}
    \Ee \left[|A^\theta X_t|^p\right]
    \leq 3^{p}\left[C_\theta^p t^{-p\theta}\eup^{-p\gamma_1t/2}|x|^p + C_{\theta,\gamma_1}^p\|F\|_\infty^p + C_{\theta,p}\right].
\end{gather*}
Hence, we obtain that for any $T>0$
\begin{gather*}
    \frac 1T\int_1^{T+1}\Ee \left[|A^\theta X_t|^p \right] \dif t
    \leq C_{\theta,p,\gamma_1, \|F\|_\infty, |x|}.
\end{gather*}
Because of \eqref{A2}, the inverse $A^{-1}$ is a compact operator and, therefore, the set
\begin{gather*}
    \Ccal_K:=\left\{x \in H\,;\,|A^\theta x| \leq K\right\}
\end{gather*}
is compact in $H$. By the Chebyshev inequality,
\begin{align*}
	\frac 1T\int_1^{T+1} P_t(x, H \setminus {\Ccal}_K)\,\dif t
    &=\frac 1T\int_1^{T+1} \Pp\left(|A^\theta X_t|>K\right) \dif t\\
	&\leq \frac 1T\int_1^{T+1} \frac{\Ee \left[|A^\theta X_t|^p\right]}{K^p}\, \dif t\\
	&\leq C_{\theta,p,\gamma_1, \|F\|_\infty, |x|} K^{-p}.
\end{align*}
This yields that $\left(\frac 1T\int_1^{T+1} P_t(x,.)\,\dif t\right)_{T>0}$ is tight and thus admits a subsequence which converges to an invariant measure, as long as the transition probability of $X_t$ has the Feller property \cite[Theorem 3.1.1]{DaZa96}.

It remains to show that $(X_t)_{t \geq 0}$ has the Feller property. By \cite[Theorem 9.29 (ii)]{PeZa07}, we have
\begin{gather*}
    \Ee \left[|X_t(x)-X_t(y)|^2\right]
    \leq
    C |x-y|^2, \quad t>0.
\end{gather*}
Fix $\delta>0$ and set $D = D_\delta = \{|X_t(x)-X_t(y)|\leq \delta\}$. Since $f$ is continuous, we can assume that $\delta=\delta(\epsilon)$ is so small that $|f(X_t(x)) - f(X_t(y))|\leq \epsilon$ on $D$. For every bounded continuous function $f$,
\begin{align*}
|P_t f(y)-P_t f(x)|&=|\Ee f(X^x_t)-\Ee f(X^y_t)| \\
& \leq \left|\Ee\left[(f(X_t(x))-f(X_t(y)))\I_{D}\right]\right| + \left|\Ee\left[(f(X_t(x))-f(X_t(y)))\I_{D^c}\right]\right|\\
& \leq \epsilon + 2\|f\|_\infty \Pp(|X_t(x)-X_t(y)| > \delta)\\
& \leq \epsilon + \frac{2C\|f\|_\infty}{\delta^2}|x-y|^2,
\end{align*}
where we use Chebyshev's inequality in the last step. Since $\epsilon>0$ is arbitrary, we see that $P_tf(y)\to P_t f(x)$ as $y\to x$, i.e.\ we have the Feller property.
\end{proof}

\section{Accessibility and an associated control problem}  \label{acc}
As before, we use the independence of $\mathds{S} = (S_t)_{t\geq 0}$ and $\mathds{W} = (W_t)_{t\geq 0}$ to represent $\Pp$ as $\Pp = \Pp^{\mathds{S}}\otimes\Pp^{\mathds{W}}$. This means that we can condition on the event $S_\cdot = \ell_\cdot$, which is an increasing c\`adl\`ag function $[0,\infty)\ni t \mapsto \ell_t$ such that $\ell_0=0$, and consider the following auxiliary equation:
\begin{equation}\label{int-e12}
    \dif X^{\ell}_t=[-AX^{\ell}_t+F(X^{\ell}_t)]\,\dif t+Q(X^{\ell}_{t-})\,\dif W_{\ell_t}, \quad X^{\ell}_{0}=x \in H,
\end{equation}
whose (unique, mild) solution \cite[Theorem 7.4]{DaZa92} is
\begin{gather} \label{int-e14}
    X^{\ell}_{t}=\eup^{-tA}x+\int_0^t \eup^{-(t-s)A} F(X^{\ell}_s)\,\dif s+ \int_0^t \eup^{-(t-s)A} Q(X^{\ell}_{s-} )\,\dif W_{\ell_s}.
\end{gather}
For any real-valued, bounded and measurable function on $H$, $f \in \Bcal_b(H,\real)$, we have
\begin{equation}\label{int-e16}
    \Ee \left(f(X_t)\right)
    = \Ee^{\mathds{S}} \left(\Ee^{\mathds{W}}\left[ f(X^{\ell}_t)\right]\big|_{\ell=S}\right).
\end{equation}
\vskip 3mm

From now on we need the following additional condition:
\begin{align}
\label{A4}\tag{\bfseries A4}
    &\parbox[t]{.9\linewidth}{There exist some $\delta>0$ with $\int_{0^+}\phi(s^{-2\delta})\,\dif s<\infty$ and $C>0$ such that for all $x\in H$
    \begin{center}
    $
        \left\|Q(x)^{-1} \eup^{-tA}\right\| \leq C t^{-\delta}.
    $
    \end{center}}
\end{align}
\begin{remark}
\eqref{A4} means that the noise is not too weak; this is necessary to guarantee accessibility of the solution to Eq.~\eqref{int-e02}. Let us illustrate this for the $\frac\alpha2$-stable subordinator case, i.e.\ for $\phi(s)=s^{\alpha/2}$, $0<\alpha<2$.  Assume, for simplicity, $Q(x)=A^\beta$ with $\beta \in \mathds R$. From \eqref{A4} we know that $\delta<1/\alpha$. By \eqref{int-e20}, we need $\beta>-1/\alpha$, which means that the strength of the noise is bounded from below. The requirement $\delta<1/\alpha$ is also consistent with \cite[Assumption 2.2 (A4)]{PrShXuZa12} for SPDEs driven by additive $\alpha$-stable noises.
\end{remark}

As before, we write $\ell:[0,\infty)\to[0,\infty)$ for a fixed trajectory of the subordinator $(S_t)_{t \geq 0}$. Since $\Pp^{\mathds{S}}$-almost all $\ell$ are strictly increasing, we can define the (generalized, right-continuous) inverse of $\ell$:
\begin{gather*}
    \ell^{-1}_t:=\inf\{s\geq 0\,:\,\ell_s>t\},\quad t\geq 0.
\end{gather*}
It is easy to check that we have for any measurable function $f:[0,\infty)\to[0,\infty)$
\begin{gather*}
    \int_{0}^{t} f\left(\ell^{-1}_{s}\right) \dif s
    = \int_{0}^{\ell^{-1}_{t}} f(s) \,\dif \ell_{s}
    \quad \text{for all $t>0$}.
\end{gather*}
See also \cite{Zhang13} for more applications of this transform.

The following proposition is crucial in order to prove the accessibility to zero of \eqref{int-e02}.
\begin{proposition}\label{acc-11}
    Assume that \eqref{A4} holds and fix $t>0$ and $m>0$. For $\Pp^{\mathds{S}}$ almost every trajectory $\ell$ of $S$ satisfying $\ell^{-1}_t \leq m$, we have
    \begin{gather}
    \label{acc-e12}
        \left|\int_0^{t} Q(z)^{-1} \eup^{-\ell^{-1}_s A} x\,\dif s\right| \leq C_1|x|,\\
    \label{acc-e14}
        \int_0^{t} \left|Q(z)^{-1} \eup^{-\ell^{-1}_s A} x\right|^2\,\dif s \leq C_2|x|^2
    \end{gather}
    for all $z \in H$, where $C_1,C_2>0$ are constants which may depend on $t$ and $m$.
\end{proposition}
\begin{proof}
    Using the Cauchy--Schwarz inequality along with \eqref{A4} we get
    \begin{align*}
        \left|\int_0^{t} Q(z)^{-1} \eup^{-\ell^{-1}_s A} x\,\dif s\right|^2
        &\leq t \int_0^{t} \left|Q(z)^{-1} \eup^{-\ell^{-1}_s A} x\right|^2\,\dif s\\
        &\leq t C^2|x|^2 \int_0^{t} (\ell_s^{-1})^{-2\delta}\,\dif s\\
        &=t C^2|x|^2 \int_0^{\ell^{-1}_t} s^{-2\delta}\,\dif \ell_s
        \leq t C^2|x|^2 \int_0^{m} s^{-2\delta}\,\dif \ell_s.
    \end{align*}
    Since $\int_{0+} \phi(s^{-2\delta})\,\dif s<\infty$,
   we know from Proposition \ref{sub-25} that $\int_0^{m} s^{-2\delta}\,\dif S_s$ is a.s.\ finite.

    This proves both \eqref{acc-e12} and \eqref{acc-e14} for $\Pp^{\mathds{S}}$-almost all trajectories $\ell$ of $S$  satisfying $\ell^{-1}_t\leq m$. Considering $t,m\in\nat$, we see that the $\Pp^{\mathds{S}}$-null can be chosen independently of $t$ and $m$.
\end{proof}

Since we consider in Eq.~\eqref{int-e02} multiplicative noise, we cannot apply the methods developed in \cite{WaXu18} and \cite{DoWaXu18} to show irreducibility. Alternatively, we resort to showing that the system \eqref{int-e02} is accessible to $0$, note that accessibility to $0$ is often used as a replacement of irreducibility when one proves ergodicity \cite{HaMa06}.

\begin{theorem} \label{acc-15}
    If the assumptions \eqref{A1}--\eqref{A4} are fulfilled and the L\'evy measure $\nu$ satisfies $\nu((0,\infty))=\infty$, then the system \eqref{int-e02} is \emph{accessible} to zero, i.e., for all $x \in H$, $\epsilon >0$, $T>0$, we have
\begin{gather*}
    \Pp\left(|X_T(x)|<\epsilon \right)>0.
\end{gather*}
\end{theorem}

\begin{proof}
Observe that for every $m\in\nat$
\begin{equation}\label{acc-e16}
    \Pp\left(|X_T(x)|<\epsilon \right)
    \geq
    \Ee^{\mathds{S}}\left[\Pp^{\mathds{W}}\left(|X^\ell_T(x)|<\epsilon \right)\middle|_{\ell=S} \I_{\{S_T \leq m\}}\right].
\end{equation}
Since $\lim_{m\to\infty}\Pp^{\mathds{S}}(S_T \leq m)=1$, we can choose $m$ in such a way that $\Pp^{\mathds{S}}(S_T\leq m)> 0$.
Because of the assumption $\nu((0,\infty))=\infty$, almost all trajectories of $S$ are strictly increasing. This observation and Lemma \ref{acc-17} below give
\begin{gather*}
    \Ee^{\mathds{S}}\left[\Pp^{\mathds{W}}\left(|X^\ell_T(x)|<\epsilon \right)\middle|_{\ell=S}\right]>0.
\qedhere
\end{gather*}
\end{proof}

\begin{lemma} \label{acc-17}
Let $\ell$ be a strictly increasing trajectory which is right continuous with left limits.
If the assumptions \eqref{A1}--\eqref{A4} hold, then the system \eqref{int-e12} is \emph{accessible} to zero, i.e., for all $x \in H$, $\epsilon >0$, $T>0$, we have for all $m\in\nat$ and $\Pp^{\mathds{S}}$-almost all $\ell$ with $\ell_T\leq m$
\begin{gather*}
    \Pp^{\mathds{W}} \left(|X^\ell_T(x)|<\epsilon \right)>0.
\end{gather*}
\end{lemma}

For the proof of Lemma \ref{acc-17} we need to study some auxiliary control problems. Consider the following problem:
\begin{equation}\label{acc-e18}
    \dif Y^\ell_t
    = [-A  Y^{\ell}_t+F(Y^{\ell}_t)]\,\dif t+Q(Y^{\ell}_{t-}) \,\dif u_{\ell_t},
\end{equation}
where $u$ is some controller which is to be determined. We say that \eqref{acc-e18} is
\begin{description}
\item[\normalfont\emph{exactly controllable}]
    if for all $x,y \in H$ and $T>0$ there exists some $u \in C([0,\ell_T];H)$ such that
    \begin{equation}\label{acc-e20}
        Y^\ell_0=x, \quad Y^\ell_T=y.
    \end{equation}
\item[\normalfont\emph{approximately controllable}]
    if for all $x,y\in H$, $T>0$ and $\epsilon >0$, there exists some $u \in C([0,\ell_T];H)$ such that
    \begin{equation}\label{acc-e22}
        Y^\ell_0=x, \quad |Y^\ell_T-y|<\epsilon,
    \end{equation}
\item[\normalfont\emph{approximately controllable to $0$}]
    if for all $x\in H$, $T>0$ and $\epsilon >0$, there exists some $u \in C([0,\ell_T];H)$ such that
    \begin{equation}\label{acc-e24}
        Y^\ell_0=x, \quad |Y^\ell_{T}|<\epsilon,
    \end{equation}
    If the target point $0$ is replaced by some fixed $y_0 \in H$, then the problem \eqref{acc-e18} is said to be \emph{approximately controllable} to $y_0$.
\end{description}

\begin{proof} [Proof of Lemma \ref{acc-17}]
Choose some $\tilde T \in (0,{1}/{\|F\|_{{\textrm{Lip}}}})$, we shall use an iteration procedure to show that there exists some
$u \in C([0, \ell_{\tilde T}], H)$ with bounded total variation such that
\begin{equation} \label{acc-e26}
\begin{cases}
    \dif \Phi^\ell_t=-A  \Phi^{\ell}_t\,\dif t+Q(Y^{\ell}_{t-})\,\dif u_{\ell_t},\\
    \Phi^\ell_0=x, \\
    \Phi^\ell_{\tilde T}=0,
\end{cases}
\end{equation}
where $Y^\ell_t$ is the solution to Eq.~\eqref{acc-e18}.
In order to keep notations simple, we drop in this proof the superscripts ``$\ell$'' of $\Phi^{\ell}$ and $Y^{\ell}$.
Because $Q$ depends on $Y$, we need to use an iteration procedure to find the controller $u$.

Define $Y^{(0)}_t \equiv x$ for all $t \in [0, \tilde T]$. We consider the following control problem:
\begin{equation} \label{acc-e28}
\begin{cases}
    \dif \Phi^{(1)}_t = -A  \Phi^{(1)}_t\,\dif t+Q(Y^{(0)}_{t-})\,\dif u^{(1)}_{\ell_t},\\
    \Phi^{(1)}_0=x,\\
    \Phi^{(1)}_{\tilde T}=0.
\end{cases}
\end{equation}
Choose
\begin{gather*}
    u^{(1)}_{\ell_t}
    = -\frac{1}{\ell_{\tilde T}} \int_0^{t} Q(Y^{(0)}_{t-})^{-1} \eup^{-sA} x\,\dif \ell_s.
\end{gather*}
It is easy to check that the assumption \eqref{A4} ensures $u^{(1)}_{\ell_t} \in H$ for all $t \geq 0$.
We have
\begin{align*}
    \Phi^{(1)}_t
    &=\eup^{-tA} x-\frac{1}{\ell_{\tilde T}} \int_0^t \eup^{-(t-s)A} Q(Y^{(0)}_{s-})  Q(Y^{(0)}_{s-})^{-1} \eup^{-sA} \,\dif \ell_s
    = \frac{\ell_{\tilde T}-\ell_{t}}{\ell_{\tilde T}} \eup^{-tA} x.
\end{align*}
Further, define
\begin{gather*}
    Y^{(1)}_t
    = \eup^{-tA} x+\int_0^t \eup^{-(t-s)A} F(Y^{(0)}_s)\,\dif s+\int_0^t \eup^{-(t-s)A} Q(Y^{(0)}_{s-})\,\dif u^{(1)}_{\ell_s};
\end{gather*}
using Proposition~\ref{acc-11} it is easy to see that
\begin{equation}\label{acc-e30}
    Y^{(1)}_t
    = \int_0^t \eup^{-(t-s)A} F(Y^{(0)}_s)\,\dif s+\Phi^{(1)}_t.
\end{equation}
For $n \in \nat$, we define recursively
\begin{equation} \label{acc-e32}
\begin{aligned}
    &u^{(n+1)}_{\ell_t}=-\frac{1}{\ell_{\tilde T}} \int_0^{t} Q(Y^{(n)}_{s-})^{-1} \eup^{-sA} x\,\dif \ell_s,\\
    &\Phi^{(n+1)}_t=\eup^{-tA} x+\int_0^t \eup^{-(t-s)A} Q(Y^{(n)}_{s-})\,\dif u^{(n+1)}_{\ell_s},\\
    &Y^{(n+1)}_t=\int_0^t \eup^{-(t-s)A} F(Y^{(n)}_s)\,\dif s+\Phi^{(n+1)}_t.
\end{aligned}
\end{equation}
The first two equalities yield
\begin{gather*}
    \Phi^{(n+1)}_t=\frac{\ell_{\tilde T}-\ell_{t}}{\ell_{\tilde T}} \eup^{-tA} x.
\end{gather*}
Therefore, we have
\begin{equation}\label{acc-e34}
\begin{aligned}
    |Y^{(n+1)}_t-Y^{(n)}_t|
    &\leq \int_0^t |\eup^{-(t-s)A} F(Y^{(n)}_s)-\eup^{-(t-s)A} F(Y^{(n-1)}_s)|\,\dif s\\
    &\leq \int_0^t \|F\|_{\textrm{Lip}} |Y^{(n)}_s-Y^{(n-1)}_s|\,\dif s.
\end{aligned}
\end{equation}
Since $\tilde T < {1}/{\|F\|_{\textrm{Lip}}}$, we see that
\begin{align*}
    \sup_{0 \leq t \leq \tilde T} |Y^{(n+1)}_t-Y^{(n)}_t| &\leq \tilde T \|F\|_{\textrm{Lip}} \sup_{0 \leq t \leq \tilde T} |Y^{(n)}_t-Y^{(n-1)}_t|\\
    &\leq\dots\leq \left(\tilde T \|F\|_{\textrm{Lip}}\right)^{n} \sup_{0 \leq t \leq \tilde T} |Y^{(1)}_t-Y^{(0)}_t|.
\end{align*}
So there exists some uniformly bounded $(Y_t)_{0 \leq t \leq \tilde T}$, which is right continuous and has left limits, such that
\begin{gather*}
    \lim_{n \to \infty} \sup_{0 \leq t \leq \tilde T} |Y^{(n)}_t-Y_t|=0.
\end{gather*}
Letting $n \to \infty$ in \eqref{acc-e32}, we obtain
\begin{align}
    &u_{\ell_t}=-\frac{1}{\ell_{\tilde T}} \int_0^{t} Q(Y_{t-})^{-1} \eup^{-sA} x\,\dif \ell_s,  \label{acc-e40} \\
    &\Phi_t=\eup^{-tA} x+\int_0^t \eup^{-(t-s)A} Q(Y_{s-})\,\dif u_{\ell_s}, \label{acc-e42} \\
    &Y_t=\int_0^t \eup^{-(t-s)A} F(Y_s)\,\dif s+\Phi_t.   \label{acc-e44}
\end{align}
From the first two equalities we see that $\Phi_{\tilde T}=0.$
For $\epsilon >0$, we get
\begin{gather*}
    |Y_{\tilde T}|
    =\left|\int_0^{\tilde T} \eup^{-(t-s)A} F(Y_s)\,\dif s\right|
    \leq \|F\| \tilde T
    \leq \frac{\epsilon}2,
    \quad \tilde T < \min\left\{\frac{\epsilon}{2 \|F\|}, \frac{1}{2\|F\|_{\textrm{Lip}}}\right\}.
\end{gather*}
Recall that
\begin{equation}\label{acc-e46}
    X_{t}=\int_0^{t} \eup^{-(t-s)A} F(X_s)\,\dif t+Z_{t},
\end{equation}
where $Z_{t}=\eup^{-tA} x+\int_0^{t} \eup^{-(t-s)A} Q(X_{s-}) \,\dif W_{\ell_s}$, then
\begin{equation}\label{acc-e48}
\begin{aligned}
    |X_{t}-Y_{t}|
    &\leq \int_0^{t}  \left|\eup^{-(t-s)A} [F(X_s)-F(Y_s)]\right|\,\dif s+|Z_{t}-\Phi_{t}|\\
    &\leq \int_0^{t} \|F\|_{\textrm{Lip}} |X_s-Y_s|\,\dif s+|Z_{t}-\Phi_{t}|.
\end{aligned}
\end{equation}
By \eqref{acc-e40}, $(u_{\ell_t})_{0 \leq t \leq T}$, as a function which is right continuous with left limit, can be embedded into a continuous function
$u \in C([0,\ell_T]; H)$ defined by
\begin{gather*}
    u_{t}
    =\frac{1}{\ell_T} \int_0^{t} Q(Y_{\ell^{-1}_s-})^{-1} \eup^{-\ell^{-1}_sA} x\,\dif s,\quad \forall t \in [0,\ell_T].
\end{gather*}
Because of \eqref{acc-e12}, the function $u_t$ is well-defined. Moreover, because of \eqref{acc-e14}, we have $\dot{u} \in L^2([0,\ell_T]; H)$.

Since $Q$ is Lipschitz by the assumption \eqref{A1},
\begin{equation}\label{acc-e50}
\begin{aligned}
    |Z_{t}-\Phi_{t}|
    &\leq \left|\int_0^t Q(X_{s})\,\dif u_{\ell_s}-\int_{0}^{t}Q(Y_{s})\,\dif u_{\ell_s}\right|+D(t,Q,X,W,u)\\
    &\leq \|Q\|_{\textrm{Lip}}\int_0^t |X_{s}-Y_{s}|\,|\dif u_{\ell_s}|+D(t,Q,X,W,u),
\end{aligned}
\end{equation}
where
\begin{align*}
    D(t,Q,X,W,u)
    &=\left|\int_0^t Q(X_{s-})\,\dif W_{\ell_s}-\int_0^t Q(X_{s-})\,\dif u_{\ell_s}\right|\\
    &=\left|\int_0^{\ell_t} Q\left(X_{\theta_s-}\right) \dif W_{s}-\int_0^{\ell_t} Q\left(X_{\theta_s-}\right) \dif u_s\right|,
\end{align*}
with $\theta_s=\ell^{-1}_s$ for all $s>0$. Hence,
\begin{align*}
    |X_{t}-Y_{t}|
    &\leq \int_0^{t}  \left|\eup^{-(t-s)A} [F(X_s)-F(Y_s)]\right|\,\dif s+|Z_{t}-\Phi_{t}|\\
    &\leq \|F\|_{\textrm{Lip}} \int_0^{t}  |X_s-Y_s|\,\dif s+\|Q\|_{\textrm{Lip}}\int_0^t |X_{s}-Y_{s}|\,|\dif u_{\ell_s}|+D^*(t,Q,X,W,u),
\end{align*}
where $D^*(t,Q,X,W,u)=\sup_{0 \leq s \leq t} D(s,Q,X,W,u)$.
As we have seen earlier, $\dot u \in L^2([0, \ell_T];H)$, and so \cite[Theorem 7.4.1]{DaZa96} yields
\begin{equation}\label{acc-e52}
    \Pp\left(D^*(t,Q,X,W,u) \leq \gamma\right)>0 \qquad \forall \gamma>0.
\end{equation}
By Gronwall's inequality (recall that $u$ has finite total variation), we have
\begin{align*}
    |X_t-Y_t|
    &\leq D^*(t, Q,X,W,u) \exp\left[\|F\|_{\textrm{Lip}} t+\|Q\|_{\textrm{Lip}}\|u\|_{\textrm{TV$[0,\ell_t]$}}\right].
\end{align*}
In view of the previous two inequalities, we obtain
\begin{gather*}
    \Pp(|X_{\tilde T}-Y_{\tilde T}|<\epsilon/2)>0.
\end{gather*}
Since $|Y_{\tilde T}|<\epsilon/2$, this further implies
\begin{equation} \label{acc-e60}
    \Pp(|X_{\tilde T}|<\epsilon)>0.
\end{equation}

Recall that the initial datum of $X$ is $x$, which can be an arbitrary point in $H$; \eqref{acc-e60} implies that the
transition probability $P_{\tilde T}(x,B(0,\epsilon))$ satisfies
\begin{gather*}
    P_{\tilde T}(x,B(0,\epsilon))>0 \qquad \forall x \in H.
\end{gather*}
Combining the Chapman--Kolmogorov equations and the above inequality, we get
\begin{gather*}
    P_{2\tilde T}(x,B(0,\epsilon))
    =\int_{H} P_{\tilde T}(y,B(0,\epsilon)) P_{\tilde T}(x,\dif y)
    > 0.
\end{gather*}
Using the above argument repeatedly, we see that for all $n \in \nat$
\begin{gather*}
    P_{n\tilde T}(x,B(0,\epsilon)) > 0.
\end{gather*}
Since $\tilde T \in (0,t_0)$ and $\epsilon >0$ are both arbitrary, $X$ is accessible to zero.
\end{proof}

\section{Galerkin approximation}\label{gal}

For every $n \in \nat$, we define an orthogonal projection $\Pi_n: H \to H_n$, where $H_n$ is the subspace of
 $H$ generated by $\{e_1,\dots,e_n\}$; that is, for any $x \in  H$ with the orthogonal expansion
 $x=\sum_{k=1}^\infty x_k e_k$, we have
$\Pi_n x=\sum_{k=1}^n x_k e_k\in H_n$.

The Galerkin approximations of Eq.~\eqref{int-e02} and Eq.~\eqref{int-e12} in $H_n$ are, respectively, as follows:
\begin{equation} \label{gal-e10}
    \dif X^n_t
    = [-AX^n_t+F^n(X^n_t)]\,\dif t+Q^n(X^n_{t-})\,\dif L^n_t,
    \quad X^n_0=x^n,
\end{equation}
and
\begin{equation} \label{gal-e12}
    \dif X^{n, \ell}_t
    = [-AX^{n, \ell}_t+F^n(X^{n, \ell}_t)]\,\dif t+Q^n(X^{n,\ell}_{t-})\,\dif W^n_{\ell_t},
    \quad X^{n,\ell}_0=x^n,
\end{equation}
where $x^n=\Pi_nx$, $F^n=\Pi_n F$, and $Q^n=\Pi_n Q \Pi_n$.

The main result of this section is the following.

\begin{theorem} \label{gal-13}
Assume that \eqref{A1} and \eqref{A3} hold.
For any $T>0$ and $\delta>0$,
\begin{equation}\label{gal-e14}
    \lim_{n \to \infty} \Pp\left(\sup_{0 \leq t\leq T} |X^n_t-X_t|>\delta\right) = 0.
\end{equation}
\end{theorem}

The proof of Theorem \ref{gal-13} relies on the following lemma.

\begin{lemma} \label{gal-15}
Assume that \eqref{A1} and \eqref{A3} hold. For all $T>0$, we have
\
\begin{equation} \label{gal-e16}
    \lim_{n \to \infty} \Ee^{\mathds{W}} \left[\sup_{0 \leq t \leq T} |X^{n,\ell}_t-X^{\ell}_t|^2\right] = 0,
\end{equation}
\end{lemma}

\begin{proof}
Throughout this proof, $C$ denotes some generic constant which may change its value from line to line. Observe that
\begin{align*}
    \dif(X^\ell_t-X^{n,\ell}_t)
    &= -A (X^\ell_t-X^{n,\ell}_t)\,\dif t+\left[F(X^\ell_t)-F^n(X^{n,\ell}_t)\right] \dif t\\
    &\qquad\mbox{}+\left[Q(X^\ell_{t-})-Q(X^{n,\ell}_{t-})+Q(X^{n,\ell}_{t-})-Q^{n}(X^{n,\ell}_{t-})\right] \dif W_{\ell_t},
\end{align*}
and
\begin{align*}
    \langle X^\ell_s-X^{n,\ell}_s, F(X^\ell_s)-F^n(X^{n,\ell}_s)\rangle
    &= \langle X^\ell_s-X^{n,\ell}_s, F(X^\ell_s)-F^n(X^{\ell}_s)\rangle\\
    &\quad\mbox{}+\langle X^\ell_s-X^{n,\ell}_s, F^n(X^{n,\ell}_s)-F^n(X^{\ell}_s)\rangle\\
    &= \langle X^\ell_s-X^{n,\ell}_s, F(X^\ell_s)-F^n(X^{\ell}_s)\rangle,
\end{align*}
where the last equality uses the fact that $X^\ell_s-X^{n,\ell}_s$ and $F(X^\ell_s)-F^n(X^{\ell}_s)$ are orthogonal. By It\^{o}'s formula, we have
\begin{equation}\label{gal-e18}
\begin{aligned}
    |X^\ell_t-X^{n,\ell}_t|^2
    &= |x-x^n|^2-2 \int_0^t |A^{\frac 12}(X^\ell_s-X^{n,\ell}_s)|^2\,\dif s\\
    &\quad\mbox{}+2 \int_0^t \langle X^\ell_s-X^{n,\ell}_s, F(X^\ell_s)-F^n(X^{\ell}_s)\rangle\,\dif s+2 \Mcal_t+[\Mcal,\Mcal]_t,
\end{aligned}
\end{equation}
where
\begin{align*}
    \Mcal_t
    &:= \int_0^t \left\langle X^\ell_{s-}-X^{n,\ell}_{s-}, \left[Q(X^\ell_{s-})-Q(X^{n,\ell}_{s-})+Q(X^{n,\ell}_{s-})-Q^{n}(X^{n,\ell}_{s-})\right] \dif W_{\ell_s}\right\rangle,\\
    [\Mcal,\Mcal]_t
    &:= \sum_{0 < s \leq t} \left|\left[Q(X^\ell_{s-})-Q(X^{n,\ell}_{s-}) +Q(X^{n,\ell}_{s-})-Q^{n}(X^{n,\ell}_{s-})\right] \Delta W_{\ell_s}\right|^2.
\end{align*}
For $t\geq 0$, set
\begin{gather*}
    \Lambda_{n,t}^\ell:=\Ee^{\mathds{W}}\left[\sup_{0 \leq s \leq t}|X^\ell_s-X^{n,\ell}_s|^2\right].
\end{gather*}
First, we have
\begin{align*}
    &\Ee^{\mathds{W}} \left[\sup_{0 \leq s\leq t} \left|\int_0^s \langle X^\ell_r-X^{n,\ell}_r, F(X^\ell_r)-F^n(X^{\ell}_r)\rangle\,\dif r\right|\right]\\
    &\leq \frac 12 \Ee^{\mathds{W}} \left[\sup_{0 \leq s\leq t} \int_0^s |X^\ell_r-X^{n,\ell}_r|^2\,\dif r\right]
    + \frac 12 \Ee^{\mathds{W}} \left[\sup_{0 \leq s\leq t} \int_0^s |F(X^\ell_r)-F^n(X^{\ell}_r)|^2  \,\dif r\right]\\
    &\leq \frac 12  \int_0^t \Ee^{\mathds{W}} \left[|X^\ell_r-X^{n,\ell}_r|^2 \right] \dif r
    + \frac 12 \int_0^t \Ee^{\mathds{W}} \left[|F(X^\ell_r)-F^n(X^{\ell}_r)|^2\right] \dif r\\
    &\leq \frac 12 \int_0^t \Lambda_{n,s}^\ell\,\dif s
    + \frac 12\int_0^t \Ee^{\mathds{W}} \left[|F(X^\ell_s)-F^n(X^{\ell}_s)|^2\right]\dif s.
\end{align*}
By the Burkholder-Davis-Gundy inequality with $p=1$ and \eqref{A1}, we obtain
\begin{align*}
    &\Ee^{\mathds{W}}\left[\sup_{0 \leq s \leq t} |\Mcal_s|\right]\\
    &\leq C \Ee^{\mathds{W}} \left[\int_0^t |X^\ell_{s-}-X^{n,\ell}_{s-}|^2 \|Q(X^\ell_{s-})-Q(X^{n,\ell}_{s-})+Q(X^{n,\ell}_{s-})-Q^{n}(X^{n,\ell}_{s-})\|^2_\mathrm{HS}\,\dif \ell_s\right]^{\frac 12}\\
    &\leq C \Ee^{\mathds{W}} \left[\sup_{0 \leq s \leq t} |X^\ell_{s-}-X^{n,\ell}_{s-}|
    \left(\int_0^t\|Q(X^\ell_{s-})-Q(X^{n,\ell}_{s-})+Q(X^{n,\ell}_{s-})-Q^{n}(X^{n,\ell}_{s-})\|^2_\mathrm{HS}\,\dif \ell_s\right)^{\frac 12}\right]\\
    &\leq \frac 12 \Lambda_{n,t}^\ell+\frac12 C^{2}\int_0^t
    \Ee^{\mathds{W}} \left[\|Q(X^\ell_{s-})-Q(X^{n,\ell}_{s-})+Q(X^{n,\ell}_{s-})-Q^{n}(X^{n,\ell}_{s-})\|^2_\mathrm{HS}\right] \dif \ell_s\\
    &\leq \frac 12 \Lambda_{n,t}^\ell
    + C^{2} \int_0^t \left(
        \Ee^{\mathds{W}} \left[|X^\ell_{s-}-X^{n,\ell}_{s-}|^2\right]
        + \Ee^{\mathds{W}} \left[\|Q(X^{n,\ell}_{s-})-Q^{n}(X^{n,\ell}_{s-})\|^2_\mathrm{HS}\right]
        \right) \dif \ell_s\\
    &\leq \frac 12 \Lambda_{n,t}^\ell
    + C^{2} \int_0^t \left(\Lambda_{n,s}^\ell
    + \Ee^{\mathds{W}} \left[\|Q(X^{n,\ell}_{s-})-Q^{n}(X^{n,\ell}_{s-})\|^2_\mathrm{HS}\right] \right) \dif \ell_s.
\end{align*}
By the It\^{o} isometry and \eqref{A1},

\allowdisplaybreaks[0]
\begin{align*}
    &\Ee^{\mathds{W}}\left[\sup_{0 \leq s \leq t} [\Mcal,\Mcal]_s\right]\\
    &= \sum_{0 < r \leq t} \Ee^{\mathds{W}} \left[\left|\left[Q(X^\ell_{r-})-Q(X^{n,\ell}_{r-})+Q(X^{n,\ell}_{r-})-Q^{n}(X^{n,\ell}_{r-})\right] \Delta W_{\ell_r}\right|^2\right]\\
    &\leq 2 \int_0^t \left(\Ee^{\mathds{W}} \left[\|Q(X^\ell_{s-})-Q(X^{n,\ell}_{s-})\|^2_\mathrm{HS}\right]
    + \Ee^{\mathds{W}} \left[\|Q(X^{n,\ell}_{s-})-Q^{n}(X^{n,\ell}_{s-})\|^2_\mathrm{HS}\right]\right) \dif \ell_s\\
    &\leq C \int_0^t  \left(\Ee^{\mathds{W}} \left[|X^\ell_{s-}-X^{n,\ell}_{s-}|^2\right]
    + \Ee^{\mathds{W}} \left[\|Q(X^{n,\ell}_{s-})-Q^{n}(X^{n,\ell}_{s-})\|^2_\mathrm{HS}\right]\right) \dif \ell_s\\
    &\leq C \int_0^t \left(\Lambda_{n,s}^\ell
    + \Ee^{\mathds{W}}\left[\|Q(X^{n,\ell}_{s-})-Q^{n}(X^{n,\ell}_{s-})\|^2_\mathrm{HS}\right]\right) \dif \ell_s.
\end{align*}
Combining the previous three inequalities with \eqref{gal-e18} and moving the term $\frac 12 \Lambda_{n,t}^\ell$ to the left hand side, we get
\begin{align*}
    \frac 12\Lambda_{n,t}^\ell \le|x-x^n|^2 + C\int_0^t \Lambda_{n,s}^\ell \,(\dif s+\dif \ell_s)
    &+C\int_0^t \Ee^{\mathds{W}} \left[\|Q(X^{n,\ell}_{s-})-Q^{n}(X^{n,\ell}_{s-})\|^2_\mathrm{HS}\right]\dif \ell_s\\
    &+C \int_0^t \Ee^{\mathds{W}}\left[|F(X^\ell_s)-F^n(X^{\ell}_s)|^2\right]\dif s.
\end{align*}
This, together with Gronwall's inequality, implies that for all $t>0$
\begin{equation}\label{gal-e20}
\begin{aligned}
    \Lambda_{n,t}^\ell
    \leq C\eup^{C(t+\ell_t)} \left(|x-x^n|^2 + \int_0^t\right. & \Ee^{\mathds{W}} \left[\|Q(X^{n,\ell}_{s-})-Q^{n}(X^{n,\ell}_{s-})\|^2_\mathrm{HS}\right]\dif \ell_s\\
    &\quad\mbox{}+\left.\int_0^t \Ee^{\mathds{W}} \left[|F(X^\ell_s) -F^n(X^{\ell}_s)|^2\right]\dif s\right).
\end{aligned}
\end{equation}
By \eqref{A1}, as $n \to \infty$,
\begin{gather*}
\begin{split}
    \|Q(x^n)-Q^{n}(x^{n})\|^2_{\mathrm{HS}}
    & \leq 4\|Q(x^{n})-Q(x)\|^2_{\mathrm{HS}} \\
    &\quad \quad+4\|\Pi_{n}(Q(x^{n})-Q(x))\Pi_{n}\|^2_{\mathrm{HS}}+4\|Q^{n}(x)-Q(x)\|^2_{\mathrm{HS}} \to 0.
\end{split}
\end{gather*}
By the dominated convergence theorem and the previous two relations, we get
\begin{gather*}
    \lim_{n \to \infty} \Lambda_{n,t}^\ell = 0,
\end{gather*}
which completes the proof.
\end{proof}

\begin{proof}[Proof of Theorem \ref{gal-13}]
By Chebyshev's inequality, for $m\in\nat$,
\begin{equation}\label{gal-e22}
\begin{aligned}
    &\Pp\left(\sup_{0 \leq t\leq T} |X^n_t-X_t|>\delta\right)
    = \Ee^{\mathds{S}}\left[\left.\Pp^{\mathds{W}}
    \left(\sup_{0 \leq t\leq T}|X^{n,\ell}_t-X^{\ell}_t|>\delta\right)\right|_{\ell=S}\right]\\
    &\qquad\qquad=
    \Ee^{\mathds{S}}\left[\I_{\{\ell_T\leq m\}}\left.\Pp^{\mathds{W}}\left(\sup_{0 \leq t\leq T}|X^{n,\ell}_t-X^{\ell}_t|>\delta\right)\right|_{\ell=S}\right]\\
    &\qquad\qquad\quad\mbox{}
    + \Ee^{\mathds{S}}\left[\I_{\{\ell_T>m\}}
    \left.\Pp^{\mathds{W}}\left(\sup_{0 \leq t\leq T}|X^{n,\ell}_t-X^{\ell}_t|>\delta\right)\right|_{\ell=S}\right]\\
    &\qquad\qquad\leq
    \delta^{-2}\Ee^{\mathds{S}}\left[\I_{\{\ell_T\leq m\}}\left.\Lambda_{n,T}^\ell\right|_{\ell=S}\right]+\Pp\left(S_T>m\right).
\end{aligned}
\end{equation}
Because of the bound \eqref{gal-e20} we can use dominated convergence. From Lemma \ref{gal-15} we get
\begin{gather*}
    \lim_{n\to\infty}\Ee^{\mathds{S}}\left[\I_{\{\ell_T\leq m\}}\left.\Lambda_{n,T}^\ell\right|_{\ell=S}\right]
    =0  \quad\text{for all $m\in\nat$}.
\end{gather*}
Since $\Pp\left(S_T>m\right)\to0$
as $m\to\infty$, we can finish the proof by
letting first $n\to\infty$ and
then $m\to\infty$ in \eqref{gal-e22}.
\end{proof}

\noindent
\textbf{Acknowledgement.} We are grateful to the associate editor and the anonymous referees for their helpful comments and professional handling. The research of L.\ Xu is supported in part by NSFC No. 12071499,  Macao S.A.R grant FDCT  0090/2019/A2 and University of Macau grant  MYRG2018-00133-FST. R.L.\ Schilling was supported through the joint Polish--German NCN--DFG `Beethoven 3' grant (NCN 2018/31/G/ST1/02252; DFG SCHI 419/11-1).

\bibliographystyle{amsplain}

\end{document}